\documentclass{jaums}

\usepackage{amssymb,amsbsy,amsthm,amsmath,graphicx,epsfig}

\usepackage{version}

\numberwithin{equation}{section}
\theoremstyle{plain}
\newtheorem{theorem}[subsection]{Theorem}
\newtheorem{proposition}[subsection]{Proposition}
\newtheorem{lemma}[subsection]{Lemma}
\newtheorem{corollary}[subsection]{Corollary}

\newtheorem{remark}[subsection]{Remark}
\renewcommand{\leq}{\leqslant}
\renewcommand{\geq}{\geqslant}
\newsavebox{\proofbox}
\savebox{\proofbox}{\begin{picture}(7,7)  \put(0,0){\framebox(7,7){}}\end{picture}}

\newcommand\Z{\mathbb{Z}}
\newcommand\R{\mathbb{R}}

\newcommand\C{\mathbb{C}}
\newcommand\N{\mathbb{N}}

\newcommand\Q{\mathbb{Q}}

\newcommand\n{\mathfrak{n}}

\newcommand\eps{\varepsilon}

\newcommand\I{{\operatorname{I}}}
\newcommand\II{{\operatorname{II}}}
\newcommand{\li}{ {\rm li } }

\begin{document}

\title{Counting the number of solutions to the Erd\H{o}s-Straus equation on unit fractions}

\author{Christian Elsholtz}
\address{Institut f\"ur Mathematik A, Steyrergasse 30/II,
Technische Universit\"at  Graz,
A-8010 Graz, Austria}
\email{elsholtz@math.tugraz.at}

\author{Terence Tao}
\address{Department of Mathematics, UCLA, Los Angeles CA 90095-1555}
\email{tao@math.ucla.edu}
\subjclass{11D68, 11N37 secondary: 11D72, 11N56}

\vspace{-0.3in}
\begin{abstract} 
For any positive integer $n$, let $f(n)$ denote the number of solutions to
the Diophantine equation 
$$\frac{4}{n}=\frac{1}{x}+\frac{1}{y}+\frac{1}{z}$$
 with $x,y,z$ positive integers. The \emph{Erd\H{o}s-Straus
conjecture} asserts that $f(n) > 0$ for every $n \geq 2$. In this paper we obtain a number of upper and lower bounds for $f(n)$ or $f(p)$ for typical values of natural numbers $n$ and primes $p$. For instance, we establish that
$$ N \log^2 N \ll \sum_{p\leq N} f(p) \ll N \log^2 N \log \log N.$$ 
These upper and lower bounds show that a typical prime has a
small number of solutions to the Erd\H{o}s-Straus Diophantine
equation; small, when compared with other additive problems, like Waring's problem.
\end{abstract}

\maketitle

\section{Introduction}
For any natural number $n \in \N = \{1,2,\ldots\}$, let $f(n)$ denote the number of solutions $(x,y,z) \in \N^3$ to the Diophantine equation 
\begin{equation}\label{xyz}
\frac{4}{n} = \frac{1}{x} + \frac{1}{y} + \frac{1}{z}
\end{equation}

(we do not assume $x,y,z$ to be distinct or in increasing order).   Thus for instance
$$ f(1)=0, f(2)=3, f(3) = 12, f(4) = 10, f(5)=12, f(6)=39, f(7)=36, 
f(8)=46,  \ldots$$
We plot the values of $f(n)$ for $n \leq 1000$, 
and separately restricting to primes $p\leq 1000$ in Figures \ref{fig1}, \ref{fig2}. 
\begin{figure}[ht]
 \includegraphics[scale=1.00]{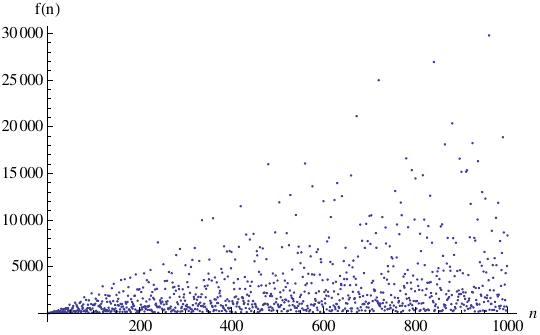}
\caption{The value $f(n)$ for all $n \leq 1000$.}\label{fig1}
\end{figure}
\begin{figure}[ht]
 \includegraphics[scale=1.00]{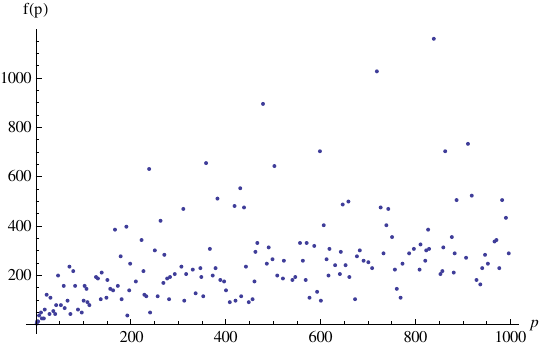}
\caption{The value $f(p)$ for all primes $p \leq 1000$.}\label{fig2}
\end{figure}

From these graphs one might be tempted to draw 
conclusions, such as ``$f(n) \gg n$
infinitely often'', that we will refute in our investigations below.

The \emph{Erd\H{o}s-Straus conjecture} (see e.g. \cite{guy}) asserts that $f(n) >
0$ for all $n \geq 2$; it remains unresolved, although there are a number of
partial results. The earliest references to this conjecture are
papers by Erd\H{o}s \cite{Erdos:1950} and Obl\'{a}th \cite{Oblath:1950}, and we
draw attention to the fact that the latter paper was submitted in 1948.

Most subsequent approaches list parametric solutions, which solve
the conjecture for $n$ lying in certain residue classes. 
These soluble classes are
either used for analytic approaches via a sieve method, or for computational
verifications.
For instance, it was shown by Vaughan \cite{vaughan} that the
number of $n < N$ for which $f(n)=0$ is at most $N \exp( - c \log^{2/3} N )$
for some absolute constant $c>0$ and all sufficiently large $N$. (Compare also
\cite{nakayama, Webb:1970, Li:1981, Yang:1982} for some weaker results).


The conjecture was  verified for all $n \leq 10^{14}$ in \cite{swett}. 
In Table \ref{table:history}
we list a more complete history of these computations,
but there may be further unpublished computations as well.

\begin{table}{\label{table:history}}
\centering
\begin{tabular}{|r|r|l|}
\hline
$5000$&$\leq$ 1950& {\text{Straus, see \cite{Erdos:1950}}}\\ \hline
$8000$&1962& {\text{Bernstein \cite{bernstein}}}\\ \hline
$20000$&$\leq$ 1969& {\text{Shapiro, see \cite{mordell}}}\\ \hline
$106128$&1948/9& {\text{Oblath \cite{Oblath:1950}}}\\ \hline
$141648$&1954&\text{Rosati \cite{rosati}}\\ \hline
$10^7$& 1964& \text{Yamomoto \cite{yamamoto}}\\ \hline
$1.1 \times 10^7$& 1976 &\text{Jollensten \cite{Jollenstein:1976}}\\ \hline
$10^8$& 1971 &\text{Terzi \cite{Terzi:1971}}\\ \hline
$10^9$ & 1994 &{\text{Elsholtz \& Roth (unpublished)}}\\ \hline
$10^{10}$&1995&\text{Elsholtz \& Roth (unpublished)}\\ \hline
$1.6 \times  10^{11}$&1996&\text{Elsholtz \& Roth (unpublished)}\\ \hline 
$10^{10}$&1999&\text{Kotsireas \cite{kotsireas}}\\ \hline
$10^{14}$&1999&\text{Swett \cite{swett}}\\ \hline
$2\times  10^{14}$&2012&\text{Bello-Hern\'{a}ndez, Benito, Fern\'{a}ndez \cite{bello}}\\ \hline
$10^{17}$ & 2014 & \text{Salez \cite{salez}} \\ \hline
\end{tabular}
\caption{Numerical verifications of the Erd\H{o}s-Straus conjecture.  It appears that Terzi's set of soluble residue classes is
  correct, but that the set of checked primes in these classes is incomplete.  Another reference to a calculation up to $10^8$ 
due to N.~Franceschine III (1978) (see \cite{guy, erdosandgraham} and frequently restated elsewhere) 
only mentions Terzi's calculation,  but is not an independent verification. We are grateful to I.~Kotsireas for
confirming this (private communication).}
\end{table}

Most of these 
previous approaches concentrated on the question whether
$f(n)>0$ or not. In this paper we will instead study the average growth or
extremal values of $f(n)$.

Since we clearly have $f(nm) \geq f(n)$ for any $n,m \in \N$, we see that to prove the Erd\H{o}s-Straus conjecture it suffices to do so when $n$ is equal to a prime $p$.  

In this paper we investigate the \emph{average} behaviour of $f(p)$ for $p$ a prime.  More precisely, we consider the asymptotic behaviour of the sum
$$ \sum_{p\leq N} f(p)$$
where $N$ is a large parameter, and $p$ ranges over all primes up to $N$.  As we are only interested in asymptotics, we may ignore the case $p=2$, and focus on the odd primes $p$.

Let us call a solution $(x,y,z)$ to \eqref{xyz} a \emph{Type I solution} if $n$ divides $x$ but is coprime to $y,z$, and a \emph{Type II solution} if $n$ divides $y,z$ but is coprime to $x$.  Let $f_\I(n), f_{\II}(n)$ denote the number of Type I and Type II solutions respectively.  By permuting the $x,y,z$ we clearly have
\begin{equation}\label{3ff}
 f(n) \geq 3f_\I(n) + 3f_\II(n)
\end{equation}
for all $n>1$.  Conversely, when $p$ is an odd prime, it is clear from considering the denominators in the Diophantine equation
\begin{equation}\label{4pn}
\frac{4}{p} = \frac{1}{x} + \frac{1}{y} + \frac{1}{z}
\end{equation}
that at least one of $x,y,z$ must be divisible by $p$; also, it is not possible for all three of $x,y,z$ to be divisible by $p$ as this forces the right-hand side of \eqref{4pn} to be at most $3/p$.  We thus have
\begin{equation}\label{4p}
f(p) = 3f_\I(p) + 3f_\II(p)
\end{equation}
for all odd primes $p$.  Thus, to understand the asymptotics of $\sum_{p \leq N} f(p)$, it suffices to understand the asymptotics of $\sum_{p \leq N} f_\I(p)$ and $\sum_{p \leq N} f_\II(p)$.  As we shall see, Type II solutions are somewhat easier to understand than Type I solutions, but we will nevertheless be able to control both types of solutions in a reasonably satisfactory manner.

We can now state our first main theorem.

\begin{theorem}[Average value of $f_\I,f_\II$]\label{main}  For all sufficiently large $N$, one has the bounds
\begin{align*}
 N \log^3 N \ll \sum_{n \leq N} f_\I(n) &\ll N \log^3 N\\
 N \log^3 N \ll \sum_{n \leq N} f_\II(n) &\ll N \log^3 N\\
N \log^2 N \ll \sum_{p \leq N} f_\I(p) &\ll N \log^2 N \log \log N \\
N \log^2 N \ll \sum_{p \leq N} f_\II(p) &\ll N \log^2 N.
\end{align*} 
\end{theorem}

Here, we use the usual asymptotic notation $X \ll Y$ or $X=O(Y)$ to denote the estimate $|X| \leq CY$ for an absolute constant $C$, and use subscripts if we wish to allow dependencies on the implied constant $C$, thus for instance $X \ll_\eps Y$ or $X = O_\eps(Y)$ denotes the estimate $|X| \leq C_\eps Y$ for some $C_\eps$ that can depend on $\eps$.   We remark that in a previous version of this manuscript, the weaker bound $\sum_{p \leq N} f_\II(p) \ll N \log^2 N \log\log N$ was claimed.  As pointed out subsequently by Jia \cite{jia}, the argument in that previous version in fact only gave $\sum_{p \leq N} f_\II(p) \ll N \log^2 N \log\log^2 N$, but can be repaired to give the originally claimed bound $\sum_{p \leq N} f_\II(n) \ll N \log^2 N \log \log N$.  These bounds are of course superceded by the results in Theorem \ref{main}.

As a corollary of this and \eqref{4p}, we see that
$$ N \log^2 N \ll \sum_{p \leq N} f(p) \ll N \log^2 N \log \log N.$$
From this, the prime number theorem, and Markov's inequality, we see that for
any $\eps > 0$, we can find a subset $A$ of primes of relative lower density at least $1-\eps$, thus
\begin{equation}\label{liminf}
 \liminf_{N \to \infty} \frac{|\{ p \in A: p \leq N\}|}{|\{ p: p \leq N \}|} \geq 1-\eps,
\end{equation}
such that $f(p) = O_\eps(\log^3 p \log \log p)$ for all $p \in A$.   Informally, a typical prime has only $O(\log^3 p \log \log p)$ solutions to the Diophantine equation \eqref{4pn}; or alternatively, for any function $\xi(p)$ of $p$ that goes to infinity as $p \to \infty$, one has $O(\xi(p) \log^3 p \log \log p)$ for all $p$ in a subset of the primes of relative density $1$.   
This may provide an explanation as to why analytic methods (such as the circle method) appear to be insufficient to resolve the Erd\H{o}s-Straus conjecture, as such methods usually only give non-trivial lower bounds on the number of solutions to a Diophantine equation in the case when the number of such solutions grows polynomially with the height parameter $N$. (There are however some exceptions to this rule, such as Gallagher's results \cite{gallagher} on representing integers as the sum of a prime and a bounded number of powers of two, but such results tend to require a large number of summands in order to compensate for possible logarithmic losses in the analysis.)

The double logarithmic factor $\log \log N$ in the above arguments arises from technical limitations to our method (and specifically, in the inefficient nature of the Brun-Titchmarsh inequality \eqref{brun} when applied to very short progressions), and we conjecture that it should be eliminated.

\begin{remark}  In view of these results, one can naively model $f(p)$ as a Poisson process with intensity at least $c \log^3 p$ for some absolute constant $c$.  Using this probabilistic model as a heuristic, one expects any given prime to have a ``probability'' $1-O(\exp(-c \log^3 p))$ of having at least one solution, which by the Borel-Cantelli lemma suggests that the Erd\H{o}s-Straus conjecture is true for all but finitely many $p$.  Of course, this is only a heuristic and does not constitute a rigorous argument.  (However, one can view the results in \cite{vaughan}, \cite{els}, based on the large sieve, as a rigorous analogue of this type of reasoning.)
\end{remark}

\begin{remark}  From Theorem \ref{main} we have the lower bound $\sum_{n \leq
    N} f(n) \gg N \log^3 N$.  In fact one has the stronger bound $\sum_{n \leq
    N} f(n) \gg N \log^6 N$ (Heath-Brown, private communication) using the methods from \cite{heath}; see Remark \ref{heath-remark} for further discussion.  Thus, for composite $n$, most solutions are in fact neither of Type I or Type II.  It would be of interest to get matching upper bounds for $\sum_{n \leq N} f(n)$, but this seems to be beyond the scope of our methods.  It would of course also be interesting to control higher moments such as $\sum_{p \leq N} f_\I(p)^k$ or $\sum_{p \leq N} f_\II(p)^k$, but this also seems to unfortunately lie out of reach of our methods, as the level of the relevant divisor sums becomes too great to handle.
\end{remark}

To prove Theorem \ref{main}, we first use some solvability criteria for Type I and Type II solutions to obtain more tractable expressions for $f_\I(p)$ and $f_\II(p)$.  As we shall see, $f_\I(p)$ is essentially (up to a factor of two) the number of quadruples $(a,c,d,f) \in \N^4$ with $4acd = p+f$, $f$ dividing $4a^2d+1$, and $acd \leq 3p/4$, while $f_\II(p)$ is essentially the number of quadruples $(a,c,d,e) \in \N^4$ with $4acde = p + 4a^2d + e$ and $acde \leq 3p/2$.  (We will systematically review the various known representations of Type I and Type II solutions in Section \ref{represent-sec}.) This, combined with standard tools from analytic number theory such as the Brun-Titchmarsh inequality and the Bombieri-Vinogradov inequality, already gives most of Theorem \ref{main}.  The most difficult bound is the upper bounds on $f_\I$, which eventually require an upper bound for expressions of the form
$$ \sum_{a \leq A} \sum_{b \leq B} \tau( kab^2 + 1 )$$
for various $A,B,k$, where $\tau(n) := \sum_{d \mid n} 1$ is the number of divisors of $n$, and $d \mid n$ denotes the assertion that $d$ divides $n$.  By using an argument of Erd\H{o}s \cite{erdos}, we obtain the following bound on this quantity:

\begin{proposition}[Average value of $\tau(kab^2+1)$]\label{kab}  For any $A, B > 1$, and any positive integer $k \ll (AB)^{O(1)}$, one has
$$ \sum_{a \leq A} \sum_{b \leq B} \tau(kab^2+1) \ll A B \log(A+B) \log(1+k).$$
\end{proposition}

\begin{remark} Using the heuristic that $\tau(n) \sim \log n$ on the average (see \eqref{tau-heuristic}), one expects the true bound here to be $O( A B \log(A+B) )$.  The $\log(1+k)$ loss can be reduced (for some ranges of $A,B,k$, at least) by using more tools (such as the Polya-Vinogradov inequality), but this slightly inefficient bound will be sufficient for our applications.
\end{remark}

We prove Proposition \ref{kab} (as well as some variants of this estimate) in Section \ref{kab-sec}.  Our main tool is a more quantitative version of a classical bound of Erd\H{o}s \cite{erdos} on the sum $\sum_{n \leq N} \tau(P(n))$ for various polynomials $P$, which may be of independent interest; see Theorem \ref{erdos-bound}.

We also collect a number of auxiliary results concerning the quantities $f_i(n)$, some of which were in previous literature.  Firstly, we have a vanishing property at odd squares:

\begin{proposition}[Vanishing]\label{odd} For any odd perfect square $n$, we have $f_\I(n)=f_\II(n) = 0$.
\end{proposition}

This observation essentially dates back to Schinzel (see
\cite{guy}, \cite{mordell}, \cite{schinzel}) and Yamomoto (see \cite{yamamoto})
and is an easy application of quadratic reciprocity \eqref{quadratic}: for the
convenience of the reader, we give the proof in Section \ref{odd-sec}.  A
variant of this proposition was also established in \cite{bello}.  Note that
this does not disprove the Erd\H{o}s-Straus conjecture, since the inequality
\eqref{3ff} does not hold with equality on perfect squares; but it does
indicate a key difficulty in attacking this conjecture, in that when showing
that $f_\I(p)$ or $f_\II(p)$ is non-zero, one can only use methods that
\emph{must necessarily fail} when $p$ is replaced by an odd square such as
$p^2$, which already rules out many strategies 
(e.g. a finite set of covering congruence strategies, or the circle method).

Next, we establish some upper bounds on $f_\I(n), f_\II(n)$ for fixed $n$:

\begin{proposition}[Upper bounds]\label{upper}  For any $n \in \N$, one has
$$ f_\I(n) \ll n^{3/5 + O(1/\log \log n)}$$
and
$$ f_\II(n) \ll n^{2/5 + O(1/\log \log n)}.$$
In particular, from this and \eqref{4p} one can conclude that for any prime $p$ one has
$$ f(p) \ll p^{3/5 + O(1/\log \log p)}.$$
\end{proposition}

This should be compared with the recent result in \cite{browning}, which gives
the bound $f(n) \ll_\eps n^{2/3+\eps}$ for all $n$ and all $\eps>0$.  For
composite $n$ the treatment of parameters dividing $n$ appears to be more
complicated and here we concentrate on those two cases that are motivated by
the Erd\H{o}s-Straus equation for prime denominator.
 
We prove
this proposition in Section \ref{upper-sec}.  

The main tools here are the multiple representations of Type I and Type II
solutions available (see Section \ref{represent-sec}) and the divisor bound
\eqref{tau-divisor}.  The values of $f(p)$ appear to fluctuate in some respects
as the values of the divisor function. The average
values  of $f(p)$ behave much more regularly. 

Moreover, in view of Theorem \ref{main}, one might also expect to have $f(n) \ll_\eps n^\eps$ for any $\eps>0$, but such logarithmic-type bounds on solutions to Diophantine equations seem difficult to obtain in general (Proposition \ref{upper} appears to be the limit of what one can obtain purely from the divisor bound \eqref{tau-divisor} alone).

In the reverse direction, we have the following lower bounds on $f(n)$ for various sets of $n$:

\begin{theorem}[Lower bounds] {\label{lowerboundturankubilius}}
For infinitely many $n$, one has
\[f(n)\geq \exp ((\log 3 +o(1))\frac{\log n}{\log \log n}),\]
where $o(1)$ denotes a quantity that goes to zero as $n \to \infty$.

For any function $\xi(n)$ going to $+\infty$ as $n \to \infty$, one has
\[f(n)\geq 
\exp \left( \frac{\log 3}{2}\log \log n -
O(\xi(n) \sqrt{\log \log n})\right) \gg (\log n)^{0.549}\]
for all $n$ in a subset $A$ of natural numbers of density $1$ (thus $|A \cap \{1,\ldots,N\}|/N \to 1$ as $N \to \infty$).

Finally, one has
\[f(p)\geq 
\exp \left( (\frac{\log 3}{2} - o(1)) \log \log p \right) \gg (\log p)^{0.549}\]
for all primes $p$ in a subset $B$ of primes of relative density $1$ (thus $|\{ p \in B: p \leq N\}|/{|\{ p: p \leq N \}|} \to 1$ as $N \to \infty$).
\end{theorem}

As the proof shows the first two lower bounds are already valid for sums of 
two unit fractions. The result directly follow 
from the growth of certain divisor functions.
An even better model for $f(n)$ is a suitable superposition of several
divisor functions. The proof will be in Section \ref{lowerboundsec-2}.

Finally, we consider (following \cite{mordell}, \cite{schinzel}) the question of finding polynomial solutions to \eqref{xyz}.  Let us call a primitive residue class $n = r \mod q$ \emph{solvable by polynomials} if there exist polynomials $P_1(n), P_2(n), P_3(n)$ which take positive integer values for all sufficiently large $n$ in this residue class (so in particular, the coefficients of $P_1,P_2,P_3$ are rational), and such that
$$ \frac{4}{n} = \frac{1}{P_1(n)} + \frac{1}{P_2(n)} + \frac{1}{P_3(n)}$$
for all $n$.  Here we recall that a residue class $r \mod q$ is \emph{primitive} if $r$ is coprime to $q$.  One could also consider non-primitive congruences, but these congruences only contain finitely many primes and are thus of less interest to solving the Erd\H{o}s-Straus conjecture (and if the Erd\H{o}s-Straus conjecture held for a common factor of $r$ and $q$, then the residue class $r \mod q$ would trivially be solvable by polynomials.

By Dirichlet's theorem, the primitive residue class $r \mod q$ contains arbitrarily large primes $p$.  For each large prime $p$ in this class, we either have one or two of the $P_1(p), P_2(p), P_3(p)$ divisible by $p$, as observed previously.  For $p$ large enough, note that $P_i(p)$ can only be divisible by $p$ if there is no constant term in $P_i$.  We thus conclude that either one or two of the $P_i(n)$ have no constant term, but not all three.  Let us call the congruence \emph{Type I solvable} if one can take exactly one of $P_1,P_2,P_3$ to have no constant term, and \emph{Type II solvable} if exactly two have no constant term.  Thus every solvable primitive residue class $r \mod q$ is either Type I or Type II solvable.

It is well-known (see \cite{Oblath:1950, mordell}) that any primitive residue class $n = r \mod 840$ is solvable by polynomials unless $r$ is a perfect square.  On the other hand, it is also known (see \cite{mordell}, \cite{schinzel}) that a primitive congruence class $n = r \mod q$ which is a perfect square, cannot be solved by polynomials (this also follows from Proposition \ref{odd}).  The next proposition essentially classifies all solvable primitive congruences.

\begin{proposition}[Solvable congruences]\label{res-class} Let $q \mod r$ be a primitive residue class.  If this class is Type I solvable by polynomials, then all sufficiently large primes in this residue class lie in one of a finite number of residue classes from one of following families:
\begin{itemize}
\item $\{ n = -f \mod 4ad\}$, where $a,d,f \in \N$ are such that $f|4a^2 d+1$. \cite{nakayama} 
\item $\{ n = -f \mod 4ac \} \cap \{n = -c/a \mod f \}$, where $a,c,f \in \N$ are such that $(4ac,f)=1$. \cite{yamamoto}
\item $\{ n = -f \mod 4cd \} \cap \{n^2 = -4c^2d \mod f \}$, where $c,d,f \in \N$ are such that $(4cd,f)=1$.
\item $\{ n = -1/e \mod 4ab \}$, where $a,b,e \in \N$ are such that $e \mid a+b$ and $(e,4ab)=1$. \cite{aigner}, \cite{rosati} 
\end{itemize}
Conversely, any residue class in one of the above four families is solvable by polynomials.

Similarly, if $q \mod r$ is Type II solvable by polynomials, then all sufficiently large primes in this residue class lie in one of a finite number of residue classes from one of the following families:
\begin{itemize}
\item $-e \mod 4ab$, where $a,b,e \in \N$ are such that $e \mid a+b$ and $(e,4ab)=1$. \cite{aigner} 
\item $-4a^2d \mod f$, where $a,d,f \in \N$ are such that $4ad \mid f+1$. \cite{vaughan}, \cite{rosati} 
\item $-4a^2d-e \mod 4ade$, where $a,d,e \in \N$ are such that $(4ad,e)=1$. \cite{nakayama} 
\end{itemize}
Conversely, any residue class in one of the above three families is solvable by polynomials. 
\end{proposition}

As indicated by the citations, mpst of these residue classes were observed to be solvable by polynomials in previous literature, but one of the conditions listed here appears to be new, and they form the essentially complete list of all such classes.  We prove Proposition \ref{res-class} in Section \ref{solution}.

\begin{remark} The results in this paper would also extend (with minor changes)
  to the more general situation in which the numerator $4$ in \eqref{4pn} is
  replaced by some other fixed positive integer, a situation considered first
  by Sierpi\'{n}ski and Schinzel (see e.g. \cite{sierpinski, vaughan,
    Palama:1958, Palama:1959, Stewart:1964}). 
\end{remark}

 We will not detail all of these extensions here but in 
Section \ref{lower-3} we extend our study of the average number of solutions to the 
more general question on sums of $k$ unit fractions 
\begin{equation}\label{m/nsumofk}
\frac{m}{n}=\frac{1}{t_1} + \frac{1}{t_2} + \cdots +
\frac{1}{t_k}.
\end{equation}
If $m \leq k$ the greedy algorithm (in this case 
also known as Fibonacci-Sylvester algorithm) shows there is a solution. Indeed, 
let $n=my+r$ with $0 <r<m$, then 
$\frac{m}{n}-\frac{1}{y+1}=\frac{m-r}{n(y+1)}$ has a smaller numerator, and
inductively a solution with $k \leq m$ is constructed.
For an alternative method (especially if $m=k=4$)
see also Schinzel \cite{schinzel:1956}.

If $m>k\geq 3$, and the $t_i$ are positive integers,
then it is an open problem if for each sufficiently large $n$ there is at least one solution. 
The Erd\H{o}s-Straus conjecture with $m=4, k=3$, discussed above,  is the most prominent case.
If  $m$ and $k$ are fixed, one can again establish sets of residue classes, such that
\eqref{m/nsumofk} is generally soluble if $n$ is in any of these residue
classes.

The problem of classifying solutions of \eqref{m/nsumofk} has been
studied by Rav \cite{rav}, S\'{o}s \cite{Sos:1905} and Elsholtz \cite{els}. 
Moreover Viola \cite{viola},
Shen \cite{shen} and Elsholtz \cite{elsholtz:2001}
have used a suitable subset of these solutions to give (for fixed $m > k\geq
3$) quantitive bounds 
on the number of those integers $n \leq N$, for which  \eqref{m/nsumofk} does not have
any solution.

In order to study, whether there is at least one solution, it is again
sufficient to concentrate on prime denominators. The average number of
solutions is smaller when averaging over the primes only, but we intend to
prove that even in the prime case the average number of solutions grows quickly, when $k$ increases.

We will focus on the case of \emph{Type II solutions}, in which
$t_2,\ldots,t_k$ are divisible by $n$.  The classification of solutions that we
give below also works for other divisibility patterns, but Type II solutions
are the easiest to count, and so we shall restrict our attention to this
case. Strictly speaking, the definition of a Type II solution here is slightly
different from that discussed previously, because we do not require that $t_1$
is coprime to $n$.  However, this coprimality is automatic when $n$ is prime (otherwise the right-hand side of \eqref{m/nsumofk} would only be at most $k/n$).  For composite $n$, it is possible to insert this condition and still obtain the lower bound \eqref{many-1}, but this would complicate the argument slightly and we have chosen not to do so here.

For given $m,k,n$, let $f_{m,k,\II}(n)$ denote the number of Type II solutions.  Our main result regarding this quantity is the following lower bound on this quantity:

\begin{theorem}\label{many-solutions}
Let $m > k\geq 3$ be fixed.  Then, for $N$ sufficiently large, one has
\begin{equation}\label{many-1}
\sum_{n \leq N} f_{m,k,\II}(n) \gg_{m,k} N (\log N)^{2^{k-1}-1}
\end{equation}
and
\begin{equation}\label{many-2}
\sum_{p \leq N} f_{m,k,\II}(p) \gg_{m,k} \frac{N (\log N)^{2^{k-1}-2}}{\log \log N}.
\end{equation}
\end{theorem}
Our emphasis here is on the exponential growth of the exponent.
In particular, as $k$ increases by one, the average number of solutions is
roughly squared.  The denominator of $\log \log N$ is present for technical reasons (due to use of the crude lower bound \eqref{phi-lower} on the Euler totient function), and it is likely that it could be eliminated (much as it is in the $m=4,k=3$ case) with additional effort.

\begin{remark} If we let $f_{m,k}(n)$ be the total number of solutions to \eqref{m/nsumofk} (not just Type II solutions), then we of course obtain as a corollary that
$$\sum_{n \leq N} f_{m,k}(n) \gg_k N (\log N)^{2^{k-1}-1}.$$
We do not expect the power of the logarithm to be sharp in this case (cf. Remark \ref{heath-remark}).  For instance, in \cite{huangandvaughan} it is shown that
$$\sum_{n \leq N} f_{m,2}(n) = \left(\frac{1}{\phi(m)}+o(1)\right) N \log^2 N$$
for any fixed $m$.
\end{remark}

Note that the equation \eqref{m/nsumofk} can be rewritten as
$$ \frac{1}{mt_1} + \cdots + \frac{1}{mt_k} + \frac{1}{-n} = 0,$$
which is primitive when $n$ is prime.  As a consequence, we obtain a lower bound for the number of integer points on the (generalised) Cayley surface:

\begin{corollary}
Let $k\geq 3$.
The number of integer points of the following generalization of Cayley's 
cubic surface, 
\[0=\sum_{i=0}^k \frac{1}{t_i},\]
with $t_i$ non-zero integers with $\min_i |t_i| \leq N$, is at least $c_k N(\log N)^{2^{k-1}-2} / \log\log N$ for some $c_k>0$ depending only on $k$.
\end{corollary}

Again, the double logarithmic factor should be removable with some additional effort, although the exponent $2^{k-1}-2$ is not expected to be sharp, and should be improvable also.

Finally, let us mention that there are many other problems on the number of solutions
of 
\begin{equation*}
\frac{m}{n}=\frac{1}{t_1} + \frac{1}{t_2} + \cdots +
\frac{1}{t_k}
\end{equation*}
which we do not study here.
Let us point to some further references:
\cite{sandor:2003},
\cite{browning}, \cite{chen-elsholtz-jiang},
\cite{elsholtz-heuberger-prodinger} study the number of 
solutions of $1$ as a sum of unit fractions.
\cite{CrootandDobbsandFriedlanderandHetzelandPappalardi} and
\cite{huangandvaughan} study the case $k=2$, also with varying numerator $m$.

Part of the first author's work on ths project was supported by the German National Merit Foundation.
The second author is supported by a grant from the MacArthur Foundation, by NSF
grant DMS-0649473, and by the NSF Waterman award.  The authors thank Nicolas
Templier for many helpful comments and references, and the referee and editor for many useful corrections and suggestions, as well as Serge Salez for pointing out an error in a previous version of the manuscript.  
The first author is very grateful to Roger Heath-Brown for very generous advice on the subject
(dating back as far as 1994).
Both authors are  particularly indebted to him for several remarks (including
Remark \ref{heath-remark}), and also for contributing some of the key arguments here 
(such as the lower bound on $\sum_{n \leq N} f_\II(n)$ and $\sum_{p \leq N}
f_\II(p)$)  which have been reproduced here with permission. 
The first author also wishes to thank Tim Browning, Ernie Croot and  Arnd Roth 
for discussions on the subject.
\section{Representation of Type I and Type II solutions}\label{represent-sec}

We now discuss the representation of Type I and Type II solutions.  There are many such representations in the literature (see e.g. \cite{aigner}, \cite{bello}, \cite{bernstein}, \cite{nakayama}, \cite{rav}, \cite{rosati}, \cite{vaughan}, \cite{webb}); we will remark how each of these representations can be viewed as a form of the one given here after describing a certain algebraic variety in coordinates.

For any non-zero complex number $n$, consider the algebraic surface
$$ S_n := \{ (x,y,z) \in \C^3: 4xyz = nyz + nxz + nxy\} \subset \C^3.$$
Of course, when $n$ is a natural number, $f(n)$ is nothing more than the number of $\N$-points $(x,y,z) \in S_n \cap \N^3$ on this surface.

It is somewhat inconvenient to count $\N$-points on $S_n$ directly, due to the fact that $x,y,z$ are likely to share many common factors.  To eliminate these common factors, it is convenient to lift $S_n$ to higher-dimensional varieties $\Sigma^\I_n$, $\Sigma^{\II}_n$ (and more specifically, to three-dimensional varieties in $\C^6$), which are adapted to parameterising Type I and Type II solutions respectively.  This will replace the three original coordinates $x,y,z$ by six coordinates $a,b,c,d,e,f$, any three of which can be used to parameterise $\Sigma^I_n$. or $\Sigma^{\II}_n$.  This multiplicity of parameterisations will be useful for many of the applications in this paper; rather than pick one parameterisation in advance, it is convenient to be able to pick and choose between them, depending on the situation.

We begin with the description of Type I solutions.
More precisely, we define $\Sigma^\I_n$ to be the set of all sextuples $(a,b,c,d,e,f) \in \C^6$ which are non-zero and obey the constraints
\begin{align}
4abd &= ne+1 \label{I-1}\\
ce &= a+b \label{I-2}\\
4abcd &= na + nb + c \label{I-3}\\
4acde &= ne + 4a^2 d + 1 \label{I-4}\\
4bcde &= ne + 4b^2 d + 1\label{I-5}\\
4acd &= n + f \label{I-6}\\
ef &= 4a^2 d + 1 \label{I-7}\\
bf &= na + c \label{I-8}\\
n^2 + 4c^2d &= f(4bcd-n).\label{I-9}
\end{align}

\begin{remark} There are multiple redundancies in these constraints; to take just one example, \eqref{I-9} follows from \eqref{I-3} and \eqref{I-6}.  One could in fact specify $\Sigma^\I_n$ using just three of these nine constraints if desired.  However, this redundancy will be useful in the sequel, as we will be taking full advantage of all nine of these identities.
\end{remark}

The identities \eqref{I-1}-\eqref{I-9} form an algebraic set that can be parameterised (perhaps up to some bounded multiplicity) by fixing three of the six coordinates $a,b,c,d,e,f$ and solving for the other three coordinates.  For instance, using the coordinates $a,c,d$, one easily verifies that
$$ \Sigma^\I_n = \left\{ (a, \frac{na+c}{4acd-n}, c, d, \frac{4a^2d+1}{4acd-n}, 4acd-n): a,c,d \in \C^3; 4acd \neq n \right\}$$
and similarly for the other $\binom{6}{3}-1 = 14$ choices of three coordinates;
we omit the elementary but tedious computations.  Thus we see that $\Sigma^\I_n$ is a three-dimensional
algebraic variety.  From \eqref{I-3} we see that the map
$$ \pi^\I_n: (a,b,c,d,e,f) \mapsto (abdn, acd, bcd)$$
maps $\Sigma^\I_n$ to $S_n$.  After quotienting out by the dilation symmetry
\begin{equation}\label{dilation}
(a,b,c,d,e,f) \mapsto (\lambda a, \lambda b, \lambda c, \lambda^{-2} d, e, f)
\end{equation}
of $\Sigma^I_n$, this map is injective.  

If $n$ is a natural number, then $\pi^I_n$ clearly maps $\N$-points of $\Sigma^I_n$ to $\N$-points of $S_n$, and if $c$ is coprime to $n$, gives a Type I solution (note that $abd$ is automatically coprime to $n$, thanks to \eqref{I-1}).  In the converse direction, all Type I solutions arise in this manner:

\begin{proposition}[Description of Type I solutions]\label{type-1}  Let $n \in \N$, and let $(x,y,z)$ be a Type I solution.  Then there exists a unique $(a,b,c,d,e,f) \in \N^6 \cap \Sigma^\I_n$ with $abcd$ coprime to $n$ and $a,b,c$ having no common factor, such that $\pi^\I_n(a,b,c,d,e,f) = (x,y,z)$.
\end{proposition}

\begin{proof} The uniqueness follows since $\pi^\I_n$ is injective after quotienting out by dilations.  To show existence, we factor  $x=ndx', y = dy', z=dz'$, where $x',y',z'$ are coprime, then after multiplying \eqref{xyz} by $ndx'y'z'$ we have
\begin{equation}\label{dxy}
4dx'y'z' = y'z' + nx'y' + nx'z'.
\end{equation}
As $y',z'$ are coprime to $n$, we conclude that $x'$ divides $y'z'$, $y'$ divides $x'z'$, and $z'$ divides $x'y'$.  Splitting into prime factors, we conclude that
\begin{equation}\label{xabc}
 x' = ab, y' = ac, z' = bc
 \end{equation}
for some natural numbers $a,b,c$; since $x',y',z'$ have no common factor, $a,b,c$ have no common factor also.  As $y,z$ were coprime to $n$, $abcd$ is coprime to $n$ also.  

Substituting \eqref{xabc} into \eqref{dxy} we obtain \eqref{I-3}, which in particular implies (as $c$ is coprime to $n$) that $c$ divides $a+b$.  If we then set $e := (a+b)/c$ and $f := 4acd - n = (na+c)/b$, then $e,f$ are natural numbers, and we obtain the other identities \eqref{I-1}-\eqref{I-9} by routine algebra.  By construction we have $\pi^\I_n(a,b,c,d,e,f) = (x,y,z)$, and the claim follows.
\end{proof}

In particular, for fixed $n$, a Type I solution exists if and only if there is an $\N$-point $(a,b,c,d,e,f)$ of $\Sigma^\I_n$ with $abcd$ coprime to $n$ (the requirement that $a,b,c$ have no common factor can be removed using the symmetry \eqref{dilation}).  By parameterising $\Sigma^\I_n$ using three or four of the six coordinates, we recover some of the known characterisations of Type I solvability:

\begin{proposition}  Let $n$ be a natural number.  Then the following are equivalent:
\begin{itemize}
\item There exists a Type I solution $(x,y,z)$.
\item There exists $a,b,e \in \N$ with $e \mid a+b$ and $4ab \mid ne+1$. \cite{aigner} 
\item There exists $a,b,c,d \in \N$ such that $4abcd = na+nb+c$ with $c$ coprime to $n$. \cite{bernstein} 
\item There exist $a,c,d,e \in \N$ such that $ne+1=4ad (ce-a)$ with $c$ coprime to $n$. \cite{rosati, mordell} 
\item There exist $a,c,d,f \in \N$ such that $n=4acd-f$ and $f \mid 4a^2d+1$, with $c$ coprime to $n$. \cite{nakayama} 
\item There exist $b,c,d,e$ with $ne = (4bcde-1)- 4b^2 d$ and $c$ coprime to $n$. \cite{bello}  
\end{itemize}
\end{proposition}

The proof of this proposition is routine and is omitted.

\begin{remark} Type I solutions $(x,y,z)$ have the obvious reflection symmetry
  $(x,y,z) \mapsto (x,z,y)$.  
With \eqref{I-6} and \eqref{I-9}
the corresponding symmetry for $\Sigma^\I_n$ is given by
$$(a,b,c,d,e,f) \mapsto \left(b,a,c,d,e,\frac{n^2+4c^2d}{f}\right).$$
We will typically only use the $\Sigma^\I_n$ parameterisation when $y \leq z$ (or equivalently when $a \leq b$), in order to keep the sizes of various parameters small.
\end{remark}

\begin{remark}  
If we consider $\N$-points $(a,b,c,d,e,f)$ of $\Sigma^\I_n$ with $a=1$, they can be explicitly parameterised as
$$ \left(1, ce-1, c, \frac{ef-1}{4}, e, f \right)$$
where $e,f$ are natural numbers with $ef=1 \mod 4$ and $n=cef-c-f$.  This shows that any $n$ of the form $cef-c-f$ with $ef=1\mod 4$ solves the Erd\H{o}s-Straus conjecture, an observation made in \cite{bello}.  However, this is a relatively small set of solutions (corresponding to roughly $\log^2 n$ solutions for a given $n$ on average, rather than $\log^3 n$), due to the restriction $a=1$.  Nevertheless, in \cite{bello} it was verified that all primes $p = 1 \mod 4$ with $p \leq 10^{14}$ were representable in this form. 
\end{remark}

Now we turn to Type II solutions.  Here, we replace $\Sigma^\I_n$ by the variety $\Sigma^\II_n$, as defined the set of all sextuples $(a,b,c,d,e,f) \in \C^6$ which are non-zero and obey the constraints
\begin{align}
4abd &= n+e \label{II-1}\\
ce &= a+b \label{II-2}\\
4abcd &= a + b + nc \label{II-3}\\
4acde &= n + 4a^2 d + e \label{II-4}\\
4bcde &= n + 4b^2 d + e\label{II-5}\\
4acd &= f+1 \label{II-6}\\
ef &= n + 4a^2 d \label{II-7}\\
bf &= nc + a \label{II-8}\\
4c^2 dn + 1 &= f(4bcd-1)\label{II-9}.
\end{align}  
This is a very similar variety to $\Sigma^\I_n$; indeed the non-isotropic dilation 
$$ (a,b,c,d,e,f) \mapsto (a, b, c/n^2, d n, n^2 e, f/n)$$
is a bijection from $\Sigma^\I_n$ to $\Sigma^\II_n$.  Thus, as with
$\Sigma^\I_n$, $\Sigma^\II_n$ is a three-dimensional algebraic variety in
$\C^6$ which can be parameterised by any three of the six coordinates in
$(a,b,c,d,e,f)$.  As before, many of the constraints can be viewed as redundant; for instance, \eqref{II-9} is a consequence of \eqref{II-3} and  \eqref{II-6}.   Note that $\Sigma^\II_n$ enjoys the same dilation symmetry
\eqref{dilation} as $\Sigma^\I_n$, and also has the reflection symmetry (using   \eqref{II-6} and  \eqref{II-9})
$$ (a,b,c,d,e,f) \mapsto \left(b,a,c,d,e,\frac{4c^2dn+1}{f}\right).$$
Analogously to $\pi^\I_n$, we have the map $\pi^\II_n: \Sigma^\II_n \to S_n$ given by
\begin{equation}\label{pi2-n}
 \pi^\II_n: (a,b,c,d,e,f) \mapsto (abd, acdn, bcdn)
\end{equation}
which is injective up to the dilation symmetry \eqref{dilation} and which, when $n$ is a natural number, maps $\N$-points of $\Sigma^\II_n$ to $\N$-points of $S_n$, and when $abd$ is coprime to $n$, gives Type II solutions.  (Note that this latter condition is automatic when $n$ is prime, since $x,y,z$ cannot all be divisible by $n$.)

We have an analogue of Proposition \ref{type-1}:

\begin{proposition}[Description of Type II solutions]\label{type-2}  Let $n \in \N$, and let $(x,y,z)$ be a Type II solution.  Then there exists a unique $(a,b,c,d,e,f) \in \N^6 \cap \Sigma^\II_n$ with $abd$ coprime to $n$ and $a,b,c$ having no common factor, such that $\pi^\I_n(a,b,c,d,e,f) = (x,y,z)$.
\end{proposition}

\begin{proof} Uniqueness follows from injectivity modulo dilations of $\pi^\II_n$ as before.  To show existence, we factor  $x=dx', y = ndy', z=ndz'$, where $x',y',z'$ are coprime, then after multiplying \eqref{xyz} by $ndx'y'z'$ we have
\begin{equation}\label{dxy-2}
4dx'y'z' = ny'z' + x'y' + x'z'.
\end{equation}
As $x'$ are coprime to $n$, we conclude that $x'$ divides $y'z'$, $y'$ divides $x'z'$, and $z'$ divides $x'y'$.  Splitting into prime factors, we again obtain the representation \eqref{xabc} for some natural numbers $a,b,c$; since $x',y',z'$ have no common factor, $a,b,c$ have no common factor also.  As $x$ was coprime to $n$, $abd$ is coprime to $n$ also.  

Substituting \eqref{xabc} into \eqref{dxy-2} we obtain \eqref{II-3}, which in particular implies that $c$ divides $a+b$.  If we then set $e := (a+b)/c$ and $f := 4acd - 1$, then $e,f$ are natural numbers, and we obtain the other identities \eqref{II-1}-\eqref{II-9} by routine algebra.  By construction we have $\pi^\II_n(a,b,c,d,e,f) = (x,y,z)$, and the claim follows.
\end{proof}

Again, we can recover some known characterisations of Type II solvability:

\begin{proposition}  Let $n$ be a natural number.  Then the following are equivalent:
\begin{itemize}
\item There exists a Type II solution $(x,y,z)$.
\item There exists $a,b,e \in \N$ with $e \mid a+b$ and $4ab \mid n+e$, and $(n+e)/4$ coprime to $n$. \cite{aigner} 
\item There exists $a,b,c,d \in \N$ such that $4abcd = a+b+nc$ with $abd$ coprime to $n$. \cite{bernstein, mordell} 
\item There exists $a,b,d \in \N$ with $4abd-1 \mid b+nc$ with $abd$ coprime to $n$. \cite{vaughan}  
\item There exist $a,c,d,e \in \N$ such that $n = (4acd-1)e - 4a^2 d$ with $(n+e)/4$ coprime to $n$. \cite{rosati} 
\item There exist $a,c,d,f \in \N$ such that $n=4ad(ce-a)-e=e(4acd-1)-4a^2d$ with $ad(ce-a)$ coprime to $n$. \cite{nakayama} 
\end{itemize}
\end{proposition}

Next, we record some bounds on the order of magnitude of the parameters $a,b,c,d,e,f$ assuming that $y \leq z$.  

\begin{lemma}\label{size}  Let $n \in \N$, and suppose that $(x,y,z) = \pi^\I_n(a,b,c,d,e,f)$ is a Type I solution such that $y \leq z$.  Then
\begin{align*}
a &\leq b \\
\frac{1}{4} n < acd &\leq \frac{3}{4} n \\
b < ce &\leq 2b\\
an \leq bf &\leq \frac{5}{3} an.
\end{align*}
If instead $(x,y,z) = \pi^\II_n(a,b,c,d,e,f)$ is a Type II solution such that $y \leq z$, then
\begin{align*}
a &\leq b \\
\frac{1}{4} n < acde &\leq n \\
b < ce &\leq 2b\\
3acd \leq f &< 4acd
\end{align*}
\end{lemma}

Informally, the above lemma asserts that the magnitudes of the quantities $(a,b,c,d,e,f)$ are controlled entirely by the parameters $(a,c,d,f)$ (in the Type I case) and $(a,c,d,e)$ (in the Type II case), with the bounds $acd \sim n, f \ll n$ in the Type I case and $acde \sim n$ in the Type II case. The constants in the bounds here could be improved slightly, but such improvements will not be of importance in our applications.

\begin{proof} First suppose we have a Type I solution.  As $y \leq z$, we have $a \leq b$.  From \eqref{I-2} we then have $b < ce \leq 2b$, and thus from \eqref{I-8} we have
$$ an \leq bf \leq an + \frac{2}{ef} bf.$$
Now, from \eqref{I-7}, $ef = 1 \mod 4$.  If $e=f=1$, then from \eqref{I-2} and
\eqref{I-8} we would have $b = na + c = na + a + b$, which is absurd, thus $ef
\geq 5$.  This gives $bf \leq 5 an/3$ as claimed.  From \eqref{I-8}
this implies that 
$c \leq 2an/3$, which in particular implies that $bcd< abdn$ and
so $y \leq z < x$.  From \eqref{xyz} we conclude that
$$ \frac{4}{3n} \leq \frac{1}{y} < \frac{4}{n}$$
which gives the bound $n/4 < acd \leq 3n/4$ as claimed.

Now suppose we have a Type II solution.  Again $a \leq b$ and $b < ce \leq 2b$.  From \eqref{II-3} we have
$$ nc < 4abcd \leq nc + 2abcd$$
and thus $n/4 < abd \leq n/2$, which by the $ce$ bound gives
$n/4 < acde \leq n$.  Since $f = 4acd-1$, we have $3acd \leq f < 4acd$, and the claim follows.
\end{proof}

\begin{remark} From the above bounds one can also easily deduce the following observation: if $4/p = 1/x + 1/y + 1/z$, then the largest denominator $\max(x,y,z)$ is always divisible by $p$.  (This observation also appears in \cite{els}.)
\end{remark}

\begin{remark}\label{heath-remark} Propositions \ref{type-1}, \ref{type-2} can be viewed as special cases of the classification by Heath-Brown \cite{heath} of primitive integer points $(x_1,x_2,x_3,x_4) \in (\Z \backslash \{0\})^4$ on Cayley's surface
$$ \left\{ (x_1,x_2,x_3,x_4): \frac{1}{x_1} + \frac{1}{x_2} + \frac{1}{x_3} + \frac{1}{x_4} = 0 \right\},$$
where by ``primitive'' we mean that $x_1,x_2,x_3,x_4$ have no common factor.  Note that if $n,x,y,z$ solve \eqref{xyz}, then $(-n, 4x, 4y, 4z)$ is an integer point on this surface, which will be primitive when $n$ is prime.  In \cite[Lemma 1]{heath} it is shown that such integer points $(x_1,x_2,x_3,x_4)$ take the form
$$ x_i = \epsilon y_j y_k y_l z_{ij} z_{ik} z_{il}$$
for $\{i,j,k,l\}=\{1,2,3,4\}$, where $\epsilon \in \{-1,+1\}$ is a sign, and the $y_i, z_{ij}$ are non-zero integers obeying the coprimality constraints
$$ (y_i,y_j) = (z_{ij},z_{kl}) = (y_i,z_{ij}) = 1$$
for $\{i,j,k,l\}=\{1,2,3,4\}$, and obeying the equation
\begin{equation}\label{da}
 \sum_{\{i,j,k,l\}=\{1,2,3,4\}} y_i z_{jk} z_{kl} z_{lj} = 0.
\end{equation}
Conversely, any $\epsilon, y_i, z_{ij}$ obeying the above conditions induces a primitive integer point on Cayley's surface.  The Type I (resp. Type II) solutions correspond, roughly speaking, to the cases when one of the $z_{1i}$ (resp. one of the $y_i$) in the factorisation
$$ n = x_1 = \epsilon y_2 y_3 y_4 z_{12} z_{13} z_{14}$$
are equal to $\pm n$.  The $y_i, z_{ij}$ coordinates are closely related to the $(a,b,c,d,e,f)$ coordinates used in this section; in \cite{heath} it is observed that these coordinates obey a number of algebraic equations in addition to \eqref{da}, which essentially describe (the closure of) the universal torsor \cite{cts} of Cayley's surface.

In \cite{heath} it was shown that the number of integer points
$(x_1,x_2,x_3,x_4)$ on Cayley's surface of maximal height
$\max(|x_1|,\ldots,|x_4|)$ bounded by $N$ was comparable to $N \log^6 N$.  This
is not quite the situation considered in our paper; a solution to \eqref{xyz}
with $n \leq N$ induces an integer point $(x_1,x_2,x_3,x_4)$ whose
\emph{minimal} height $\min(|x_1|,\ldots,|x_4|)$ is bounded by $N$.
Nevertheless, the results in \cite{heath} can be easily modified (by minor
adjustments to account for the restriction that three of the $x_i$ are
positive, and restricting $n$ to be a multiple of $4$ to eliminate divisibility
constraints) to give a \emph{lower bound} $\sum_{n \leq N} f(n) \gg N \log^6 N$
for the number of such points, though it is not immediately obvious whether
this lower bound can be matched by a corresponding upper bound. Nevertheless,
we see that there are several logarithmic factors separating the general
solution count from the Type I and Type II solution count; in particular, for
generic $n$, the majority of solutions to \eqref{xyz} will neither be Type I
nor Type II.
In spite of this, the number of Type I and Type II solutions is the relevant
quantity for studying the Erd\H{o}s-Straus conjecture, as it is naturally
to study it for prime denominators only. 
\end{remark}

We close this section with a small remark on the
well known standard classification of solutions in Mordell's book:
His two cases (in his notation)
$$\frac{m}{p} = \frac{1}{abd} + \frac{1}{acd} + \frac{1}{bcdp}$$
with $(a,b)=(a,c)=(b,c)=1$ and $p\nmid abcd$ and
$$\frac{m}{p} = \frac{1}{abd} + \frac{1}{acdp} + \frac{1}{bcdp}$$
$(a,b)=(a,c)=(b,c)=1$ with $p\nmid abd$ suggest that $p\mid c$ might 
be possible. Here we prove, for $m>3$ and $p$ coprime to $m$, 
that none of the denominators can be divisible by $p^2$.
In particular $p \nmid abcd$ in both of the cases above.  

\begin{proposition}{\label{refinement}}
Let $m/p = 1/x + 1/y +1/z$ where $m > 3$, 
$p$ is a prime not dividing $m$, and $x,y,z$ are natural numbers.  
Then none of $x,y,z$ are divisible by $p^2$.
\end{proposition}

Note that there are a small number of counterexamples to this proposition for $m \leq 3$, such as ${3}/{2} = {1}/{1} + {1}/{4} + {1}/{4}$.

\begin{proof}  We may assume that $(x,y,z)$ is either a Type I or Type II solution (replacing $4$ by $m$ as needed).  In the Type I case $(x,y,z) = (abdp, acd, bcd)$, the claim is already clear since $abcd$ is known to be coprime to $p$.  In the Type II case $(x,y,z) = (abd, acdp, bcdp)$ it is known that $abd$ is coprime to $p$, so the only remaining task is to establish that $c$ is coprime to $p$ also.

Suppose $c$ is not coprime to $p$; then $y, z$ are both divisible by $p^2$.  In particular
$$ \frac{1}{y}+\frac{1}{z} \leq \frac{2}{p^2}$$
and hence
$$ \frac{m}{p} > \frac{1}{x} \geq \frac{m}{p} - \frac{2}{p^2}.$$
Taking reciprocals, we conclude that
$$ p < mx \leq p (1 - \frac{2}{mp})^{-1}.$$
Bounding $(1-\eps)^{-1} < 1+2\eps$ when $0 < \eps < 1/2$, we conclude that
$$ p < mx < p + \frac{4}{m}.$$
But if $m>3$, this forces $mx$ to be a non-integer, a contradiction.
\end{proof}

\section{Upper bounds for $f_i(n)$}\label{upper-sec}

We may now prove Proposition \ref{upper}.

We begin with the bound for $f_\I(n)$.  By symmetry we may restrict attention to Type I solutions $(x,y,z)$ for which $y \leq z$.  By Proposition \ref{type-1} and Lemma \ref{size}, these solutions arise from sextuples $(a,b,c,d,e,f) \in \N^6 \cap \Sigma^\I_n$ obeying the Type I bounds in Lemma \ref{size}.  In particular we see that
$$ e \cdot f \cdot (cd)^2 \cdot ac = (acd)^2 (\frac{ce}{b}) (\frac{bf}{a}) \ll n^3,$$
and hence at least one of $e, f, cd, ac$ is $O(n^{3/5})$.

Suppose first that $e \ll n^{3/5}$.   For fixed $e$, we see from \eqref{I-1} and the divisor bound \eqref{tau-divisor} that there are $n^{O({1}/{\log\log n})}$ choices for $a,b,d$, giving a net total of $n^{3/5+O({1}/{\log\log n})}$ points in $\Sigma^\I_n$ in this case.

Similarly, if $f \ll n^{3/5}$, \eqref{I-6} and the divisor bound gives $n^{O({1}/{\log\log n})}$ choices for $a,c,d$ for each $f$, giving $n^{3/5+O({1}/{\log\log n})}$ solutions.  If $cd \ll n^{3/5}$, one uses \eqref{I-9} and the divisor bound to get $n^{O({1}/{\log\log n})}$ choices for $b,f,c,d$ for each choice of $cd$, and if $ac \ll n^{3/5}$, then \eqref{I-8} and the divisor bound gives $n^{O({1}/{\log\log n})}$ choices for $a,b,c,f$ for each fixed $ac$.  Putting all this together (and recalling that any three coordinates in $\Sigma^\I_n$ determine the other three) we obtain the first part of Proposition \ref{upper}.

Now we prove the bound for $f_\II(n)$, which is similar.  Again we may restrict attention to sextuples $(a,b,c,d,e,f) \in \N^6 \cap \Sigma^\II_n$ obeying the Type II bounds in Lemma \ref{size}.  In particular we have
$$ e^2 \cdot (ad) \cdot (ac) \cdot (cd) = (acde)^2 \leq n^2$$ 
and so at least one of $e, ad, ac, cd$ is $O(n^{2/5})$.

If $e \ll n^{2/5}$, we use \eqref{II-1} and the divisor bound to get $n^{O({1}/{\log \log n})}$ choices for $a,b,d$ for each $e$.  If $ad \ll n^{2/5}$, we use \eqref{II-7} and the divisor bound to get $n^{O({1}/{\log \log n})}$ choices for $a,d,e,f$ for each fixed $ad$.  If $ac \ll n^{2/5}$, we use \eqref{II-8} to get $n^{O({1}/{\log \log n})}$ choices for $a,c,b,f$ for each fixed $ac$.  If $cd \ll n^{2/5}$, we use \eqref{II-9} and the divisor bound to get $n^{O({1}/{\log \log n})}$ choices for $b,c,d,f$ for each fixed $cd$.  Putting all this together we obtain the second part of Proposition \ref{upper}.

\begin{remark}  This argument, together with the fact that a large number $n$
  can be factorised in expected $O(n^{o(1)})$ time (using, say, the quadratic
  sieve \cite{pomerance}), gives an algorithm to find all Type I solutions for
  a given $n$ in expected run time $O(n^{3/5+o(1)})$, and an algorithm to find all the Type II solutions in expected run time $O(n^{2/5+o(1)})$.
\end{remark}

\section{Insolubility for odd squares}\label{odd-sec}

We now prove Proposition \ref{odd}.  Suppose for contradiction that $n$ is an odd perfect square (in particular, $n=1 \mod 8$) with a Type I solution.
Then by Proposition \ref{type-1}, we can find an $\N$-point $(a,b,c,d,e,f)$ in $\Sigma^\I_n$.

Let $q$ be the largest odd factor of $ab$.  From \eqref{I-1} we have $ne+1=0 \mod q$.  Since $n$ is a perfect square, we conclude that
$$
\left( \frac{e}{q} \right) = \left( \frac{-1}{q} \right) = (-1)^{(q-1)/4}$$
thanks to \eqref{quadratic-1}. Since $n=1 \mod 8$, we see from \eqref{I-1} that $e=3 \mod 4$.  By quadratic reciprocity \eqref{quadratic} we thus have
$$
\left( \frac{q}{e} \right) = 1.$$ 
On the other hand, from \eqref{I-2} we see that $ab = -a^2 \mod e$, and thus 
$$ \left( \frac{ab}{e} \right) = \left( \frac{-1}{e} \right) = -1$$
by \eqref{quadratic-1}.  This forces $ab \neq q$, and so (by definition of $q$) $ab$ is even.  By \eqref{I-1}, this forces $e=7 \mod 8$, which by \eqref{quadratic-2} implies that 
$$\left( \frac{2}{e} \right) =1$$ 
and thus 
$$\left( \frac{q}{e} \right) = \left( \frac{ab}{e} \right),$$ 
a contradiction.

The proof in the Type II case is almost identical, using \eqref{II-1}, \eqref{II-2} in place of \eqref{I-1}, \eqref{I-2}; we omit the details.

\section{Lower bounds I}\label{lower-sec}

Now we prove the lower bounds in Theorem \ref{main}.

We begin with the lower bound
\begin{equation}\label{2n-lower}
 \sum_{n \leq N} f_\II(n) \gg N \log^3 N.
\end{equation}

Suppose $a,c,d,e$ are natural numbers with $d$ square-free, $e$ coprime to $ad$, $e > a$, and $acde \leq N/4$.  Then the quantity
\begin{equation}\label{acde}
 n := 4acde - e - 4a^2 d
\end{equation}
is a natural number of size at most $N$, and $(a,ce-a,c,d,e,4acd-1)$ is an $\N$-point of $\Sigma^\II_\n$.  Applying $\pi^\II_n$, we obtain a solution
$$ (x,y,z) = (a(ce-a)d, acdn, (ce-a)cdn)$$
to \eqref{xyz}.  We claim that this is a Type II solution, or equivalently that $a(ce-a)d$ is coprime to $n$.  As $e$ is coprime to $ad$, we see from
\eqref{acde} that $n$ is coprime to $ade$, so it suffices to show that $n$ is coprime to $b := ce-a$.  But if $q$ is a common factor of both $n$ and $b$, then from the identity \eqref{II-8} (with $f = 4acd-1$) we see that $q$ is also a common factor of $a$, a contradiction.  Thus we have obtained a Type II solution.  Also, as $d$ is square-free, any two quadruples $(a,c,d,e)$ will generate different solutions, as the associated sextuples $(a,ce-a,c,d,e,4acd-1)$ cannot be related to each other by the dilation \eqref{dilation}.  Thus, it will suffice to show that there are at least $\delta N \log^3 N$ quadruples $(a,c,d,e) \in \N$ with $d$ square-free, $e$ coprime to $ad$, $e > a$, and $acde \leq N/4$ for some absolute constant $\delta>0$.  Restricting $a,c,d$ to be at most $N^{0.1}$ (say), we see that the number of possible choices of $e$ is at least $\delta' ({N}/{acd}) {\phi(ad)}/{ad}$, where $\phi$ is the Euler totient function and $\delta'>0$ is another absolute constant.  It thus suffices to show that
$$ \sum_{a,c,d \leq N^{0.1}} \mu^2(d) \frac{\phi(ad)}{ad} \frac{1}{adc} \gg \log^3 N,$$
where $\mu$ is the M\"obius function (so $\mu^2(d)=1$ exactly when $d$ is square-free).  Using the elementary estimate $\phi(ad) \geq \phi(a) \phi(d)$ and factorising, we see that it suffices to show that
\begin{equation}\label{doe}
 \sum_{d \leq N^{0.1}} \frac{\mu(d)^2 \phi(d)}{d^2} \gg \log N.
\end{equation}
But this follows from Lemma \ref{upper-crude}.

Now we prove the lower bound
$$ \sum_{n \leq N} f_\I(n) \gg N \log^3 N,$$
which follows by a similar method.

Suppose $a,c,d,f$ are natural numbers with $d$ square-free, $f$ dividing $4a^2d+1$ and coprime to $c$, $d \geq f$, and $acd \leq N/4$.  Then the quantity
\begin{equation}\label{acdef}
n := 4acd - f
\end{equation}
is a natural number which is at most $N$, and $(a, b, c, d, {4a^2 d+1}/{f}, f)$ is an $\N$-point of $\Sigma^\I_n$, where
$$ b := c \frac{4a^2d+1}{f}-e = \frac{na+c}{f}.$$
Applying $\pi^\I_n$, this gives a solution
$$ (x,y,z) = (abdn, acd, bcd)$$
to \eqref{xyz}, and as before the square-free nature of $d$ ensures that each
quadruple $(a,c,d,f)$ gives a different solution.  We claim that this is a Type
I solution, i.e. that $abcd$ is coprime to $n$.  As $f$ divides $4a^2d+1$,
$f$ and with \eqref{acdef} also $n$ is coprime to $ad$. As  $f$ and $c$ are
coprime by assumption, $n$ is coprime to $acd$ by \eqref{acdef}.  
As $b = (na+c)/f$, we conclude that $n$ is also coprime to $b$.

Thus it will suffice to show that there are at least $\delta N \log^3 N$ quadruples $(a,c,d,f) \in \N^4$ with $f$ coprime to $2ac$, and $d$ square-free with $f$ dividing $4a^2d+1$, $d \geq f$, and $acd \leq N/4$, for some absolute constant $\delta>0$.

We restrict $a, c, f$ to be at most $N^{0.1}$.  If $f$ is coprime to $2ac$, then there is a unique primitive residue class of $f$ such that $4a^2d+1$ is a multiple of $f$ for all $d$ in this class.  Also, there are at least $\delta {N}/{acf}$ elements $d$ of this residue class with $d \geq f$ and $acd \leq N/4$ for some absolute constant $\delta>0$; a standard sieving argument shows that a positive proportion of these elements are square-free.  Thus, we have a lower bound of
$$ \sum_{a,c,f \leq N^{0.1}: (f,2ac)=1} \frac{N}{acf} $$
for the number of quadruples.  Restricting $f$ to be odd and then using the crude sieve
\begin{equation}\label{crude-sieve}
 1_{(f,2ac)=1} \geq 1 - \sum_p 1_{p \mid f} 1_{p \mid a} - \sum_p 1_{p \mid f} 1_{p \mid c}
\end{equation}
where $p$ ranges over odd primes, where $1_E$ denotes the indicator function of a statement $E$ (i.e. $1_E=1$ if $E$ holds, and $1_E=0$ otherwise), one easily verifies that the above expression is at least $\delta N \log^3 N$ for some absolute constant $\delta>0$, and the claim follows.

Now we establish the lower bound
$$ \sum_{p \leq N} f_\II(p) \gg N \log^2 N.$$
We will repeat the proof of \eqref{2n-lower}, but because we are now counting primes instead of natural numbers we will need to invoke the Bombieri-Vinogradov inequality at a key juncture.

Suppose $a,c,d,e$ are natural numbers with $d$ square-free, $a,c,d \leq N^{0.1}$, and $e$ between $N^{0.6}$ and $N/4acd$ with
\begin{equation}\label{pacd}
 p := 4acde - e - 4a^2 d
\end{equation}
prime.  Then $p$ is at most $N$ and at least $N^{0.6}$, and in particular is automatically coprime to $ade$ (and thus $ce-a$, by previous arguments).  Thus, as before, each such $(a,c,d,e)$ gives a Type II solution for a prime $p \leq N$, with different quadruples giving different solutions.  Thus it suffices to show that there are at least $\delta N \log^2 N$ quadruples $(a,c,d,e)$ with the above properties for some absolute constant $\delta>0$.

Fix $a,c,d$.  As $e$ ranges from $N^{0.6}$ to $N/4acd$, the expression
\eqref{pacd} traces out a primitive residue class modulo $4acd-1$, omitting at
most $O(N^{0.6})$ members of this class that are less than $N$.  Thus, the
number of primes of the form \eqref{pacd} for fixed $acd$ is
$$\pi(N;4acd-1,-4a^2d) - O(N^{0.6}), $$
where $\pi(N;q,t)$ denotes the number of primes $p<N$ that are congruent to $t$ mod $q$.
We replace $\pi(N;4acd-1,-4a^2d)$ by a good approximation, and bound the error.
If we set
$$D(N;q) := \max_{(a,q)=1} \left|\pi(N;q,a) - \frac{\li(N)}{\phi(q)}\right|$$
(as in \eqref{dnq}), where $\li(x) := \int_0^x {dt}/{\log t}$ is the Cauchy
principal value of the logarithmic integral, the number of primes of the form 
 \eqref{pacd} for fixed $acd$ is at least
$$ \frac{\li(N)}{\phi(4acd-1)} - D(N;4acd-1) - O(N^{0.6})$$
The overall contribution of those $acd$ combinations referring to the $O(N^{0.6})$
error term is at most 
$O( (N^{0.1})^3 N^{0.6} ) = o(N \log^2 N)$, while $\li(N)$ is comparable to $N/\log N$, so it will suffice to show the lower bound
\begin{equation}\label{lower-ii} \sum_{a,c,d \leq N^{0.1}}  \frac{\mu^2(d)}{\phi(4acd-1)} \gg \log^3 N
\end{equation}
and the upper bound
\begin{equation}\label{upper-ii} \sum_{a,c,d \leq N^{0.1}} D(N;4acd-1) = o( N \log^2 N ).
\end{equation}
We first prove \eqref{lower-ii}.  Using the trivial bound $\phi(4acd-1) \leq 4acd$, it suffices to show that
$$ \sum_{a,c,d \leq N^{0.1}} \frac{\mu^2(d) }{acd} \gg \log^3 N$$
which upon factorising reduces to showing
$$ \sum_{d \leq N^{0.1}} \frac{\mu^2(d)}{d} \gg \log N.$$
But this follows from Lemma \ref{upper-crude}.

Now we show \eqref{upper-ii}.  Writing $q := 4acd-1$, we can upper bound the left-hand side of \eqref{upper-ii} somewhat crudely by
$$ \sum_{q \leq N^{0.3}} D(N;q) \tau(q+1)^2.$$
From divisor moment estimates (see \eqref{tau-2}) we have
$$ \sum_{q \leq N^{0.3}} \frac{\tau(q+1)^4}{q} \ll \log^{O(1)} N;$$
hence by Cauchy-Schwarz, we may bound the preceding quantity by
$$ \ll \log^{O(1)} N \left(\sum_{q \leq N^{0.3}} q D(N;q)^2\right)^{1/2}.$$
Using the trivial bound $D(N;q) \ll N/q$, we bound this in turn by
$$ \ll N^{1/2} \log^{O(1)} N \left(\sum_{q \leq N^{0.3}} D(N;q)\right)^{1/2}.$$
But from the Bombieri-Vinogradov inequality \eqref{bombieri}, we have
$$ \sum_{q \leq N^{0.3}} D(N;q) \ll_A N \log^{-A} N$$
for any $A > 0$, and the claim \eqref{upper-ii} follows.

Finally, we establish the lower bound 
$$ \sum_{p \leq N} f_\I(p) \gg N \log^2 N.$$
Unsurprisingly, we will repeat many of the arguments from preceding cases.
Suppose $a,c,d,f$ are natural numbers with $a,c,f \leq N^{0.1}$ with $(a,c)=(2ac,f)=1$, $N^{0.6} \leq d \leq N/4ac$, such that $f$ divides $4a^2d+1$, and the quantity
\begin{equation}\label{acdef-p}
p := 4acd - f
\end{equation}
is prime.  Then $p$ is at most $N$ and is at least $N^{0.4}$, and in particular is coprime to $a,c,f$; from \eqref{acdef-p} it is coprime to $d$ also.  This thus yields a Type I solution for $p$; by the coprimality of $a,c$, these solutions are all distinct as no two of the associated sextuples $(a, b, c, d, {4a^2 d+1}/{f}, f)$ can be related by \eqref{dilation}.  Thus it suffices to show that there are at least $\delta N \log^2 N$ quadruples $(a,c,d,f)$ with the above properties for some absolute constant $\delta>0$.

For fixed $a,c,f$, the parameter $d$ traverses a primitive congruence class
modulo $f$, and $p=4acd-f$ traverses 
a primitive congruence class modulo $4acf$, that omits at most $O(N^{0.6})$ of
the elements of this class that are less than $N$.  
By \eqref{dnq}, the total number of $d$ that thus give a prime $p$ for fixed $acf$ is at least
$$ \frac{\li(N)}{\phi(4acf)} - D(N;4acf) - O(N^{0.6})$$
and so by arguing as before it suffices to show the bounds
$$ \sum_{a,c,f \leq N^{0.1}} 1_{(a,c) = (2ac,f)=1} \frac{1}{\phi(4acf)} \gg \log^3 N$$
and
$$ \sum_{a,c,f \leq N^{0.1}} D(N;4acf) = o( N \log^2 N).$$
But this is proven by a simple modification of the arguments used to establish \eqref{upper-ii}, \eqref{lower-ii} (the constraints $(a,c)=(2ac,f)=1$ being easily handled by an elementary sieve such as \eqref{crude-sieve}).  This concludes all the lower bounds for Theorem \ref{main}.

\section{Lower bounds II}{\label{lowerboundsec-2}}

Here we prove Theorem \ref{lowerboundturankubilius}.

\begin{proof}
For any natural numbers $m,n$, let $g_2(m,n)$ denote the number of solutions $(x,y)\in \N^2$ to the Diophantine equation 
 ${m}/{n}={1}/{x}+{1}/{y}$.  Since
$$ \frac{1}{x} + \frac{1}{y} = \frac{1}{x} + \frac{1}{2y} + \frac{1}{2y}$$
we conclude the crude bound $f(n) \geq g_2(4,n)$ for any $n$.  

In \cite[Theorem 1]{browning} it was shown that $g_2(m,n) \gg 3^s$ whenever $n$
is the product of $s$ distinct primes congruent to $-1 \mod m$.  
Since $g_2(m, kn) \geq g_2(m, n)$ for any $k$, we conclude that
\begin{equation}\label{flower}
f(n) \geq g_2(4,n) \gg 3^{w_4(n)}
\end{equation}
for all $n$, where $w_m(n)$ is the number of distinct prime factors of $n$ that
are congruent 
to $-1 \mod m$.

Now we prove the first part of the theorem.  Let $s$ be a large number, and let $n$ be the product of the first $s$ primes equal to $-1 \mod 4$, then from the prime number theorem in arithmetic progressions we have $\log n = (1+o(1)) s \log s$, and thus $s = (1+o(1)) {\log n}/{\log \log n}$.  From \eqref{flower} we then have
$$ f(n) \gg \exp\left( \log 3 (1+o(1)) \frac{\log n}{\log \log n}\right).$$
Letting $s \to \infty$ we obtain the claim.

For the second part of the theorem, we use the Tur\'{a}n-Kubilius inequality 
(Lemma \ref{turan-kubilius}) to the additive function $w_4$.  This inequality gives that
$$\sum_{n \leq N} |w_4(n)- \frac{1}{2}\log \log N|^2 \ll N \log \log N.$$
From this and Chebyshev's inequality (see also \cite[p. 307]{tenenbaum}), we see that 
$$ w_4(n) \geq \frac{1}{2}\log \log n +O(\xi(n) \sqrt{\log \log n})$$
for all $n$ in a density $1$ subset of $\N$.  The claim then follows from \eqref{flower}.

Now we turn to the third part of the theorem.  We first deal with the case when $p = 4t-1$ is prime, then
$$ \frac{4}{p} = \frac{4}{p+1} + \frac{1}{t(4t-1)}$$
which in particular implies that
$$ f(p) \geq g_2(4,p+1)$$
and thus
$$ f(p) \gg 3^{w_4(p+1)}.$$
By Lemma \ref{barban-turankubilius} we know that
\begin{equation}\label{w4}
 w_4(p+1) \geq \left(\frac{1}{2}-o(1)\right) \log \log p
 \end{equation}
for all $p$ in a a set of primes of relative prime density $1$.

It remains to deal with those primes $p$ congruent to $1 \mod 4$.  Writing
$$ \frac{4}{p} = \frac{1}{(p+3)/4} + \frac{3}{p(p+3)/4}$$
we see that
$$ f(p) \geq g_2(3,p(p+3)/4) \gg 3^{w_3((p+3)/4)} \gg 3^{w_3(p+3)}.$$
It thus suffices to show that
$$ w_3(p+3) \geq \left(\frac{1}{2}-o(1)\right) \log \log p$$
for all $p$ in a set of primes of relative density $1$.  
But this can be established by the same techniques used to establish
\eqref{w4}. 

\end{proof}

\section{Sums of divisor functions}\label{kab-sec}

Let $P: \Z \to \Z$ be a polynomial with integer coefficients, which for simplicity we will assume to be non-negative, and consider the sum
$$ \sum_{n \leq N} \tau(P(n)).$$
In \cite{erdos}, Erd\H{o}s established the bounds
\begin{equation}\label{oed}
N \log N \ll_{P} \sum_{n \leq N} \tau(P(n)) \ll_{P} N \log N 
\end{equation}
for all $N>1$ and for $P$ irreducible; note that the implied constants here can depend on both the degree and the coefficients of $P$.  This is of course consistent with the heuristic $\tau(n) \sim \log n$ ``on average''.  Of course, the irreducibility hypothesis is necessary as otherwise $P(n)$ would be expected to have many more divisors.

In this section we establish a refinement of the Erd\H{o}s upper bound that
gives a more precise description of the dependence of the implied constant on
$P$ (and with irreducibility replaced by a much weaker hypothesis), 
which may be of some independent interest:

\begin{theorem}[Erd\H{o}s-type bound]\label{erdos-bound}  Let $N > 1$, let $P$ be a polynomial with degree $D$ and coefficients being non-negative integers of magnitude at most $N^l$.  For any natural number $m$, let $\rho(m)$ be the number of roots of $P \mod m$ in $\Z/m\Z$, and suppose one has the bound
\begin{equation}\label{pj}
\rho( p^j ) \leq C
\end{equation}
for all primes $p$ and all $j \geq 1$.  Then
$$ N \sum_{m \leq N} \frac{\rho(m)}{m} \ll \sum_{n \leq N} \tau(P(n)) \ll_{D,l,C} N \sum_{m \leq N} \frac{\rho(m)}{m}.$$
\end{theorem}

\begin{remark} For any fixed $P$, one has \eqref{pj} for some $C = C_P$ (by many applications of Hensel's lemma, and treating the case of small $p$ separately), and when $P$ is irreducible one can use tools such as Landau's prime ideal theorem to show that $\sum_{m \leq N} {\rho(m)}/{m} \ll_P \log N$ (indeed, much more precise asymptotics are available here).  
See \cite{stewart} for more precise bounds on $C$ in terms of quantities such as the discriminant $\Delta(P)$ of $P$; bounds of this type go back to Nagell \cite{nagell} and Ore \cite{ore} (see also \cite{sandor}, \cite{huxley}).  One should in fact be able to establish a version of Theorem \ref{erdos-bound} in which the implied constant depends explicitly on the $\Delta(P)$ rather than on $C$ by using the estimates of Henriot \cite{henriot} (which build upon earlier work of Barban-Vehov \cite{barban}, Daniel \cite{daniel}, Shiu \cite{shiu}, Nair \cite{nair}, and Nair-Tenenbaum \cite{nairt}), but we will not do so here, as we will need to apply this bound in a situation in which the discriminant may be large, but for which the bound $C$ in \eqref{pj} can still be taken to be small.  However, the version of Nair's estimate given in \cite[Theorem 2]{breteche}, having no explicit dependence on the discriminant, may be able to give an alternate derivation of Theorem \ref{erdos-bound}; we thank the referee for this observation.

Thus we see that Erd\H{o}s' original result \eqref{oed} is a corollary of Theorem \ref{erdos-bound}.  For special types of $P$ (e.g. linear or quadratic polynomials), more precise asymptotics on $\sum_{n \leq N} \tau(P(n))$ are known (see e.g. \cite{fouvry}, \cite{fi} for the linear case, and \cite{hooley}, \cite{scourfield},
\cite{mckee}, \cite{mckee2}, \cite{mckee3} for the quadratic case), but the
methods used are less elementary (e.g. Kloosterman sum bounds in the linear
case, and class field theory in the quadratic case), and do not cover all
ranges of coefficients of $P$ for the applications to the Erd\H{o}s-Straus
conjecture.  See also \cite{pom} for another upper bound in the quadratic case
which is uniform over large ranges of coefficients but gives weaker bounds 
(losing some powers of $\log N$).
\end{remark}

\begin{proof}
Our argument will be based on the methods in \cite{erdos}.   In this proof all
implied constants will be allowed to depend on $D,l$ and $C$.

We begin with the lower bound, which is very easy.  Clearly
\begin{equation}\label{low}
 \tau(P(n)) \geq \sum_{m \leq N: m \mid P(n)} 1
\end{equation}
and thus
$$ \sum_{n \leq N} \tau(P(n)) \geq \sum_{m \leq N} \sum_{n \leq N: m \mid P(n)} 1.$$
The expression $P(n) \mod m$ is periodic in $n$ with period $m$, and thus for $m \leq N$ one has
\begin{equation}\label{period}
N \frac{\rho(m)}{m} \ll \sum_{n \leq N: m \mid P(n)} 1 \ll N \frac{\rho(m)}{m}
\end{equation}
which gives the lower bound on $\sum_{n \leq N} \tau(P(n))$.

Now we turn to the upper bound, which is more difficult.  We first establish a preliminary bound
\begin{equation}\label{note}
 \sum_{n \leq N} \tau(P(n))^2 \ll N \log^{O(1)} N
\end{equation}
using an argument of Landreau \cite{landreau}.  Let $n \leq N$.  By the coefficient bounds on $P$ we have
\begin{equation}\label{n-bound}
P(n) \ll N^{O(1)}.
\end{equation}
Using the main lemma from \cite{landreau}, we conclude that
$$ \tau(P(n))^2 \ll \sum_{m \leq N: m \mid P(n)} \tau(m)^{O(1)}$$
and thus
$$ \sum_{n \leq N} \tau(P(n))^2 \ll \sum_{m \leq N} \tau(m)^{O(1)} \sum_{n \leq N: m \mid P(n)} 1.$$
Using \eqref{pj}, we may crudely bound $\sum_{n \leq N: m \mid P(n)} 1 \leq \tau(m)^{O(1)}$, thus
$$ \sum_{n \leq N} \tau(P(n))^2 \ll \sum_{m \leq N} \tau(m)^{O(1)}$$
and the claim then follows from Lemma \ref{upper-crude}.

In view of \eqref{note} and the Cauchy-Schwarz inequality, we may discard from the $n$ summation any subset of $\{1,\ldots,N\}$ of cardinality at most $N \log^{-C'} N$ for sufficiently large $C'$.  We will take advantage of this freedom in the sequel.

Suppose for the moment that we could reverse \eqref{low} and obtain the bound
\begin{equation}\label{dpn}
\tau(P(n)) \ll \sum_{m \leq N: m \mid P(n)} 1.
\end{equation}
Combining this with \eqref{period}, we would obtain
\begin{align*}
 \sum_{n \leq N} \tau(P(n)) &\ll \sum_{m \leq N} \sum_{n \leq N: m \mid P(n)} 1 \\
 &\ll \sum_{m \leq N} \frac{N}{m} \rho(m)
\end{align*}
which would give the theorem.  Unfortunately, while \eqref{dpn} is certainly true when $P(n) \leq N^2$, it can fail for larger values of $P(n)$, and from the coefficient bounds on $P$ we only have the weaker upper bound \eqref{n-bound}.

Nevertheless, as observed by Erd\H{o}s, we have the following substitute for \eqref{dpn}:

\begin{lemma}\label{leo}  Let $C'$ be a fixed constant.  For all but at most $O(N \log^{-C'} N)$ values of $n$ in the range $1 \leq n \leq N$, either \eqref{dpn} holds, or one has
$$
\tau(P(n)) \ll O(1)^r \sum_{m \in S_r: m \mid P(n)} 1
$$
for some $2 \leq r \ll (\log \log N)^2$, where $S_r$ is the set of all $m$ with the following properties:
\begin{itemize}
\item $m$ lies between $N^{1/4}$ and $N$.
\item $m$ is $N^{1/r}$-smooth (i.e. $m$ is divisible by any prime larger than $N^{1/r}$).
\item $m$ has at most $(\log \log N)^2$ prime factors.
\item $m$ is not divisible by any prime power $p^k$ with $p \leq N^{1/2}$, $k > 1$, and $p^k \geq N^{1/8(\log \log N)^2}$.
\end{itemize}
\end{lemma}

The point here is that the exponential loss in the $O(1)^r$ factor will be more than compensated for by the $N^{1/r}$-smooth requirement, which as we shall see gains a factor of $r^{-cr}$ for some absolute constant $c>0$.

\begin{proof}  The claim follows from \eqref{dpn} when $P(n) \leq N^2$, so we may assume that $P(n) > N^2$.  

We factorise $P(n)$ as
$$ P(n) = p_1 \ldots p_J$$
where the primes $p_1 \leq \ldots \leq p_J$ are arranged in non-decreasing order.  Let $0 \leq j < J$ be the largest integer such that $p_1 \ldots p_j \leq N$.  If $j=0$ then all prime factors of $P(n)$ are greater than $N$, and thus by \eqref{n-bound} we have $J=O(1)$ and thus $\tau(P(n)) = O(1)$, which makes the claim \eqref{dpn} trivial.  Thus we may assume that $j \geq 1$.

Suppose first that all the primes $p_{j+1},\ldots,p_J$ have size at least $N^{1/2}$.  Then from \eqref{n-bound} we in fact have $J = j+O(1)$, and so
$$ \tau(P(n)) \ll \tau(p_1 \ldots p_j).$$
Note that every factor of $p_1 \ldots p_j$ divides $P(n)$ and is at most $N$, which gives \eqref{dpn}.  Thus we may assume that $p_{j+1}$, in particular, is less than $N^{1/2}$, which forces
\begin{equation}\label{naples} N^{1/2} < p_1 \ldots p_j \leq N
\end{equation}
and $p_j <N^{1/2}$.

Following \cite{erdos}, we eliminate some small exceptional sets of natural numbers $n$.  First we consider those $n$ for which $P(n)$ has at least $(\log \log N)^2$ distinct prime factors.  For such $P(n)$, one has $\tau(P(n)) \geq 2^{(\log \log N)^2}$, which is asymptotically larger than any given power of $\log N$; thus by \eqref{note}, the set of such $n$ has size at most $O( N \log^{-C'} N )$ and can be discarded.

Next, we consider those $n$ for which $P(n)$ is divisible by a prime power $p^k$ with $p \leq N^{1/2}$, $k > 1$, and $p^k \geq N^{1/8(\log \log N)^2}$.  By reducing $k$ if necessary we may assume that $p^k \leq N$. For each $p$ and $k$, there are at most $O( ({N}/{p^k}) \rho(p^k) ) = O( {N}/{p^k} )$ numbers $n$ with $P(n)$ divisible by $p^k$, thanks to \eqref{pj}; thus the total number of such $n$ is bounded by
$$ \ll N \sum_{p \leq N^{1/2}} \sum_{j \geq 2: p^j \geq N^{1/8(\log \log N)^2}} \frac{1}{p^j}$$
which can easily be computed to be $O( N \log^{-C'} N)$.  Thus we may discard all $n$ of this type.

After removing all such $n$, we must have $p_j > N^{1/8(\log \log N)^2}$.  Indeed, after eliminating the exceptional $n$ as above, $p_1 \ldots p_j$ is the product  of at most $(\log \log N)^2$ prime powers, each of which is bounded by $N^{1/8(\log\log N)^2}$, or is a single prime larger than $N^{1/8(\log \log N)^2}$.  The former possibility thus contributes at most $N^{1/8}$ to the final product $p_1 \ldots p_j$; from \eqref{naples} we conclude that the latter possibility must occur at least once, and the claim follows.

Let $r$ be the positive integer such that
$$ N^{1/(r+1)} < p_j \leq N^{1/r},$$
then $2 \leq r \ll (\log \log N)^2$.  The primes $p_{j+1},\ldots,p_J$ have size at least $N^{1/(r+1)}$, so by \eqref{n-bound} we have $J = j + O(r)$, which implies that
$$ \tau(P(n)) \ll O(1)^r \tau(p_1 \ldots p_j).$$
As $p_1\ldots p_j$ is at least $N^{1/2}$, we have
$$ \tau(p_1 \ldots p_j) 
\leq 2 \sum_{m \mid p_1 \ldots p_j; m \geq (p_1 \ldots p_j)^{1/2}} 1
\leq 2 \sum_{m \mid p_1 \ldots p_j; m \geq N^{1/4}} 1.$$
Note that all $m$ in the above summand lie in $S_r$ and divide $P(n)$.  The claim follows.
\end{proof}

Invoking the above lemma, it remains to bound
\begin{align*}
&\sum_{m \leq N} \sum_{n \leq N: m \mid P(n)} 1  
\quad + \sum_{r=2}^{O((\log \log N)^2)} O(1)^r \sum_{m \in S_r} \sum_{n \leq N: m \mid P(n)} 1.
\end{align*}
by $O( N \sum_{n \leq N} {P(m)}/{m} )$.  The first term was already shown to be acceptable by \eqref{period}.  For the second sum, we also apply \eqref{period} and bound it by
\begin{equation}\label{modo}
 \ll N \sum_{r=2}^{O((\log \log N)^2)} O(1)^r \sum_{m \in S_r}  \frac{\rho(m)}{m}.
 \end{equation}
To estimate this expression, let $r, m$ be as in the above summation, and factor $m$ into primes.  As in the proof of Lemma \ref{leo}, the contribution to $m$ coming from primes less than $N^{1/8(\log\log N)^2}$ is at most $N^{1/8}$, and the primes larger than $N^{1/8(\log\log N)^2}$ that divide $m$ are distinct.  Hence, by the pigeonhole principle (as in \cite{erdos}), there exists $t \geq 1$ with $r2^t \ll (\log \log N)^2$ such that the $N^{1/r}$-smooth number $m$ has at least $\lfloor{rt}/{100}\rfloor$ distinct prime factors between $N^{1/2^{t+1} r}$ and $N^{1/2^t r}$, and can thus be factored as $m = q_1 \ldots q_{\lfloor{rt}/{100}\rfloor} u$ where $q_1 < \ldots < q_{\lfloor{rt}/{100}\rfloor}$ are primes between $N^{1/2^{t+1} r}$ and $N^{1/2^t r}$, and $u$ is an integer of size at most $N$.  From the Chinese remainder theorem and \eqref{pj} we have the crude bound
$$ \rho(m) \ll O(1)^{rt} \rho(u)$$
and thus
$$ \sum_{m \in S_r}  \frac{\rho(m)}{m} \ll \sum_{t=1}^\infty O(1)^{rt} \frac{1}{\lfloor\frac{rt}{100}\rfloor!} \left(\sum_{N^{1/2^{t+1}r} \leq p \leq N^{1/2^t r}} \frac{1}{p}\right)^{\lfloor{rt}/{100}\rfloor} \sum_{u \leq N} \frac{\rho(u)}{u}.$$
By the standard asymptotic $\sum_{p<x} {1}/{p} = \log\log x + O(1)$, we have
$$\sum_{N^{1/2^{t+1}r} \leq p \leq N^{1/2^t r}} \frac{1}{p} = O(1);$$
putting this all together, we can bound \eqref{modo} by
$$ \ll \left(\sum_{r=2}^\infty \sum_{t=1}^\infty \frac{O(1)^{rt}}{\lfloor\frac{rt}{100}\rfloor!}\right) \sum_{m \leq N} \frac{\rho(m)}{m}$$
and the claim follows.
\end{proof}

We isolate a simple special case of Theorem \ref{erdos-bound}, when the polynomial $P$ is linear:

\begin{corollary}\label{eb}  If $a, b, N$ are natural numbers with $a,b \ll N^{O(1)}$, then
$$ \sum_{n \leq N} \tau(an+b) \ll \tau( (a,b) ) N \log N$$
where $(a,b)$ is the greatest common divisor of $a$ and $b$.
\end{corollary}

\begin{proof}  By the elementary inequality $\tau(nm) \leq \tau(n) \tau(m)$ we may factor out $(a,b)$ and assume without loss of generality that $a,b$ are coprime. 

We apply Theorem \ref{erdos-bound} with $P(n) := an+b$.  From the coprimality of $a,b$ and elementary modular arithmetic, we see that $\rho(m) \leq 1$ for all $m$, and the claim follows.
\end{proof}

We may now prove Proposition \ref{kab} from the introduction.

\begin{proof}[Proof of Proposition \ref{kab}]
We divide into two cases, depending on whether $A \geq B$ or $A \leq B$.

First suppose that $A \geq B$.  From Corollary \ref{eb} we have
$$ \sum_{a \leq A} \tau(kab^2+1) \ll A \sum_{m \leq A} \frac{1}{m} \ll A \log A,$$
for each fixed $b \leq B$, and the claim follows on summing in $B$.  (Note that this argument in fact works whenever $A \geq B^\eps$ for any fixed $\eps>0$.)

Now suppose that $A \leq B$.  
For each fixed $a \in A$, we apply Theorem \ref{erdos-bound} to the polynomial $P_{ka}(b) := kab^2+1$.  To do this we first must obtain a bound on $\rho_{ka}(p^j)$, where $\rho_{ka}(m)$ is the number of solutions $b \mod m$ to $ka b^2+1=0 \mod m$.  Clearly $\rho_{ka}(m)$ vanishes whenever $m$ is not coprime to $ka$, so it suffices to consider $\rho_{ka}(p^j)$ when $p$ does not divide $ka$.  Then $P_{ka}$ is quadratic, and a simple application of Hensel's lemma reveals that $\rho_{ka}(p^j) \leq 2$ for all odd prime powers $p^j$ and $\rho_{ka}(p^j) \leq 4$ for $p=2$.  We may therefore apply Theorem \ref{erdos-bound} and conclude that
$$ \sum_{b \leq B} \tau(kab^2+1) \ll B \sum_{m \leq B} \frac{\rho_{ka}(m)}{m}.$$
It thus suffices to show that
\begin{equation}\label{ab}
 \sum_{a \leq A} \sum_{m \leq B} \frac{\rho_{ka}(m)}{m} \ll A \log B \log(1+k).
\end{equation}

To control $\rho_{ka}(m)$, the obvious tool to use here is the quadratic
reciprocity law \eqref{quadratic}.  To apply this law, it is of course
convenient to first reduce to the case when $a$ and $m$ are odd.    If $m = 2^j m'$ for some odd $m'$, then $\rho_{ka}(m) \ll \rho_{ka}(m')$, and from this it is easy to see that the bound \eqref{ab} follows from the same bound with $m$ restricted to be odd.  Similarly, by splitting $a = 2^l a'$ and absorbing the $2^l$ factor into $k$ (and dividing $A$ by $2^l$ to compensate), we may assume without loss of generality that $a$ is odd.

As previously observed, $\rho_{ka}(m)$ vanishes unless $ka$ and $m$ are coprime, so we may also restrict to the case $(ka,m)=1$, where $(n,m)$ denotes the greatest common divisor of $n,m$.  If $p$ is an odd prime not dividing $ka$, then from elementary manipulation and Hensel's lemma we see that
$$ \rho_{ka}(p^j) = \rho_{ka}(p) \leq 1 + \left( \frac{-ka}{p} \right),$$
and thus for odd $m$ coprime to $ka$ we have
$$ \rho_{ka}(m) \leq \prod_{p \mid m} \left(1 + \left( \frac{-ka}{p} \right)\right).$$
For odd $m$, not necessarily coprime to $ka$, we thus have
$$ \rho_{ka}(m) \leq \prod_{p \mid m; (p,2ka)=1} \left(1 + \left( \frac{-ka}{p} \right)\right).$$
using the multiplicativity properties of the Jacobi symbol, one has
$$ 1 + \left( \frac{-ka}{p} \right) \leq \sum_{j:p^j \mid m} \left( \frac{-ka}{p^j} \right)$$
whenever $p \mid m$ and $(p,2ka)=1$, and thus
$$ \rho_{ka}(m) \leq \prod_{p \mid m; (p,2ka)=1} \sum_{j:p^j \mid m} \left( \frac{-ka}{p^j} \right).$$
The right-hand side can be expanded as
$$ \sum_{q \mid m; (q,2ka)=1} \left( \frac{-ka}{q} \right).$$
We can thus bound the left-hand side of \eqref{ab} by
$$ \sum_{q \leq B: (q,2k)=1} \sum_{a \leq A; (a,2q)=1} \left( \frac{-ka}{q} \right)  
\sum_{m \leq B; q \mid m} \frac{1}{m}.$$
The final sum is of course $({\log \frac{B}{q}})/{q}+O({1}/{q})$.  The contribution of the error term is bounded by
$$ O( \sum_{q \leq B} \sum_{a \leq A} \frac{1}{q} ) = O( A \log B )$$
which is acceptable, so it suffices to show that
\begin{equation}\label{qo}
 \left|\sum_{q \leq B: (q,2k)=1} \sum_{a \leq A; (a,2q)=1} \left( \frac{-ka}{q} \right) \frac{\log \frac{B}{q}}{q}\right| \ll A \log B \log(1+k).
\end{equation}
We first dispose of an easy contribution, when $q$ is less than $A$.  The expression 
$$a \mapsto \left( \frac{-ka}{q} \right) 1_{(a,2q)=1}$$ 
is periodic with period $2q$ and sums to zero (being essentially a quadratic character on $\Z/2q\Z$), and so in this case we have
$$ \sum_{a \leq A; (a,2q)=1} \left( \frac{-ka}{q} \right) = O( q ).$$
One could obtain better estimates and deal with somewhat larger $q$ here by using tools such as the P\'olya-Vinogradov inequality, but we will not need to do so here; similarly for the treatment of the regime $A \leq q \leq kA$ below.  In any event, the contribution of the $q < A$ case is bounded by
$$ O\left( \sum_{q \leq A} q \frac{\log \frac{B}{q}}{q} \right) = O( A \log B )$$
which is acceptable.

Next, we deal with the contribution when $q$ is between $A$ and $kA$.  Here we crudely bound the Jacobi symbol in magnitude by $1$ and obtain a bound of
$$ O( \sum_{A \leq q \leq kA} \sum_{a \leq A} \frac{\log B}{q} ) = O( A \log B \log(1+k) )$$
which is acceptable.

Finally, we deal with the case when $q$ exceeds $kA$.  We write $k = 2^m k'$ where $k'$ is odd, then from quadratic reciprocity \eqref{quadratic} (and \eqref{quadratic-1}, \eqref{quadratic-2}) we have
$$ \left( \frac{-ka}{q} \right) = c(q) \left( \frac{q}{k'a} \right)$$
where $c(q) := (-1)^{(q-1)/2 + m(q^2-1)/8}$ is periodic with period $8$.  We can thus rewrite this contribution to \eqref{qo} as
$$ \left|\sum_{a \leq A; (a,2)=1} \sum_{kA \leq q \leq B: (q,2ak)=1} c(q) \left( \frac{q}{k'a} \right) \frac{\log\frac{B}{q}}{q}\right|.$$
For any fixed $a$ in the above sum, the expression 
$$q \mapsto c(q) \left( \frac{q}{k'a} \right) 1_{(q,2ak)=1}$$
 is periodic with period $8k'a = O(kA)$, is bounded in magnitude by $1$ and has mean zero.  A summation by parts then gives
$$ \left|\sum_{kA \leq q \leq B: (q,2ak)=1} c(q) \left( \frac{q}{k'a} \right) \frac{\log\frac{B}{q}}{q}\right| \ll \log B$$
and so on summing in $A$ we see that this contribution is acceptable.  This concludes the proof of the proposition.
\end{proof}

We now record some variants of Proposition \ref{kab} that will also be useful in our applications.

\begin{proposition}[Average value of $\tau_3(ab+1)$]\label{ab3}  For any $A, B > 1$, one has
\begin{equation}\label{abba}
 \sum_{a \leq A} \sum_{b \leq B} \tau_3(ab+1) \ll A B \log^2(A+B).
\end{equation}
\end{proposition}

\begin{proof}  By symmetry we may assume that $A \leq B$, so that $ab \ll B^2$ for all $a \leq A$ and $b \leq B$.  For any $n$, $\tau_3$ is the number of ways to represent $n$ as the product $n=d_1 d_2 d_3$ of three terms.  One of these terms must be at most $n^{1/3}$, and so
$$ \tau_3(n) \ll \sum_{d \mid n: d \leq n^{1/3}} \tau(\frac{n}{d}).$$
We can thus bound the left-hand side of \eqref{abba} by
$$ \ll \sum_{d \ll B^{2/3}} \sum_{a \leq A} \sum_{b \leq B: d \mid ab+1} \tau(\frac{ab+1}{d}).$$
Note that for fixed $a,d$, the constraint $d \mid ab+1$ is only possible if $a$ is coprime to $d$, and restricts $b$ to some primitive residue class $q \mod d$ for some $q = q_{a,d}$ between $1$ and $d$.    Writing $b = cd+q$, we can thus bound the above expression by
$$ \ll \sum_{d \ll B^{2/3}} \sum_{a \leq A} \sum_{c \ll B/d} \tau( ac + r )$$
where $r = r_{a,d} := ({aq+1})/{d}$.  Note that $r$ is clearly coprime to $a$.  Thus by Corollary \ref{eb}, we may bound the preceding expression by
$$ \ll \sum_{d \ll B^{2/3}} \sum_{a \leq A} \frac{B}{d} \log B$$
which is $O(AB \log^2 B)$.  The claim follows.
\end{proof}

\begin{proposition}[Average value of $\tau(ab+cd)$]\label{abcd-prop}  For any $A,B,C,D > 1$, one has
\begin{equation}\label{abcd}
\sum_{\substack{a \leq A, b \leq B, c \leq C, d \leq D: \\ (a,b,c,d) = 1}} \tau(ab+cd) \ll ABCD \log(A+B+C+D).
\end{equation}
\end{proposition}

\begin{proof} By symmetry we may assume that $A,B,C \leq D$.  Then for fixed $a,b,c$ coprime, we have
$$ \sum_{d \leq D} \tau(ab+cd) \ll D \log D$$
by Corollary \ref{eb}, and the claim follows by summing in $a,b,c,d$.
\end{proof}

\begin{remark}  Informally, one can view the above propositions as asserting 
that the heuristics $\tau(n) \ll \log n$, $\tau_3(n) \ll \log^2 n$ are valid 
on average (in a first moment sense) on the range of various polynomial forms 
in several variables.  A result similar to Proposition \ref{abcd-prop} 
was established in \cite[Lemma 3]{heath}, but with the coprimality condition 
$(a,b,c,d)=1$ replaced by $(ab,cd)=1$, and also the divisor function 
$\tau$ being restricted by forcing one of the divisors 
to live in a given dyadic range, with the logarithm being removed as a 
consequence.  
Also, products of three factors were permitted instead of the terms $ab, cd$.  As remarked after \cite[Lemma 4]{heath}, the logarithmic term in \eqref{abcd} is necessary.
\end{remark}

\section{Upper bound for $\sum_{n \leq N} f_\I(n)$ and $\sum_{p \leq N} f_\I(p)$}

Now that we have established Proposition \ref{kab}, we can obtain upper bounds on sums of $f_\I$.

We begin with the bound
$$ \sum_{n \leq N} f_\I(n) \ll N \log^3 N.$$
By Proposition \ref{type-1} and symmetry followed by Lemma \ref{size}, it suffices to show that there are at most $O(N \log^3 N)$ septuples $(a,b,c,d,e,f,n) \in \N^7$ obeying \eqref{I-1}-\eqref{I-9} and the Type I estimates from Lemma \ref{size}.  In particular, $acd \ll N$, $f$ is a factor of $4a^2 d+1$, and $n = 4acd-f$.  As $a,c,d,f$ determine the remaining components of the septuple, we may thus bound the number of such septuples as
$$ \sum_{a,c,d: acd \ll N} \tau(4a^2 d + 1).$$
Dividing $a,c,d$ into dyadic blocks ($A/2 \leq a \leq A$, etc.) and applying Proposition \ref{kab} (with $k=4$) to each block, we obtain the desired bound $O( N \log^3 N)$.

Now we establish the bound
$$ \sum_{p \leq N} f_\I(p) \ll N \log^2 N \log \log N.$$
As before, it suffices to count quadruples $(a,c,d,f)$ with $acd \ll N$, and $f$ a factor of $4a^2d+1$; but now we can restrict $p = 4acd-f$ to be prime.  Also, from Proposition \ref{type-1} we may assume that $p$ is coprime to $acd$ (and hence to $4acd$, if we discard the prime $p=2$).  

Thus we may assume without loss of generality that $-f \mod 4ad$ is a primitive residue class.  From the Brun-Titchmarsh inequality \eqref{brun}, we conclude that for each fixed $a,d,f$, there are $O( {N}/({\phi(4ad) \log(N/4ad)}) )$ primes $p$ in this residue class that are less than $N$ if $ad \leq N/100$ (say); if instead $ad > N/100$, then we of course only have $O(1) = O( {N}/{\phi(4ad)})$ primes in this class.  Thus, in any event, we can bound the number of such primes as $O( {N}/({\phi(4ad) \log(2 + N/ad)}) )$.  We therefore have the bound
\begin{equation}\label{logo}
\sum_{p \leq N} f_\I(p) \ll \sum_{a,d: ad \ll N} \tau(4a^2 d+1) \frac{N}{\phi(4ad) \log(2 + N/ad)}.
\end{equation}
By dyadic decomposition (and bounding $\phi(4ad) \geq \phi(ad)$), it thus suffices to show that
\begin{equation}\label{phoad}
 \sum_{a,d: N/2 \leq ad \leq N} \frac{\tau(4a^2 d + 1)}{\phi(ad)} \ll \log^2 N.
\end{equation}
Indeed, assuming this bound for all $N$, we can bound the right-hand side of \eqref{logo} by
$$ \sum_{j=1}^{O(\log N)} \frac{N \log^2 N}{j} \ll N \log^2 N \log \log N$$
and the claim follows.

To prove \eqref{phoad}, we would like to again apply Proposition \ref{kab}, but we must first deal with the $\phi(ad)$ denominator.  From \eqref{phiq-bound} one has
$$ \frac{1}{\phi(ad)} \ll \frac{1}{ad} \sum_{s \mid a} \sum_{t \mid d} \frac{1}{st}.$$
Writing $a = sa'$, $d = td'$, we may thus bound the left-hand side of \eqref{phoad} by
$$
\ll \frac{1}{N} \sum_{s,t: st \leq N} \frac{1}{st} \sum_{a',d': a'd' \leq N/st} \tau(4s^2 t (a')^2 d' + 1).$$
Applying Proposition \ref{kab} to the inner sum (decomposed into dyadic blocks, and setting $k = 4s^2 t$), we see that
$$ \sum_{a',d': a'd' \leq N/st} \tau(4s^2 t (a')^2 d' + 1) \ll \frac{N}{st} \log^2 \frac{N}{st} \log(1+s^2 t).$$
Inserting this bound and summing in $s,t$ we obtain the claim.

\section{Upper bound for $\sum_{n \leq N} f_\II(n)$ and $\sum_{p \leq N} f_\II(p)$}

Now we prove the upper bound 
$$ \sum_{n \leq N} f_\II(n) \ll N \log^3 N.$$
By Proposition \ref{type-2} followed by Lemma \ref{size} (and symmetry), it
suffices to show that there are at most $O(N \log^3 N)$ $\N$-points
$(a,b,c,d,e,f)$ that lie in $\Sigma_n^\II$ for some $n \leq N$, which also
obeys the Type II bound $acde\leq N$ in Lemma \ref{size}. 

Observe from \eqref{II-1}-\eqref{II-9} that $a,c,d,e$ determine the other
variables $b,f,n$.  Thus, it suffices to show that there are $O(N \log^3 N)$
quadruples $(a,b,d,e) \in \N^4$ with $acde \leq N$.  But this follows from \eqref{tau-k} with $k=4$.

Finally, we prove the upper bound
$$ \sum_{p \leq N} f_\II(p) \ll N \log^2 N.$$
By dyadic decomposition, it suffices to show that
\begin{equation}\label{ppn}
 \sum_{N/2 \leq p \leq N} f_\II(p) \ll N \log^2 N.
\end{equation}
As before, we can bound the left-hand side (up to constants) by the number of quadruples $(a,c,d,e) \in \N^4$ with $acde \ll N$.  However, by \eqref{II-4}, we may also add the restriction that $4acde-4a^2 d - e$ is a prime between $N/2$ and $N$.   Also, if we set $b := ce-a$, then by Lemma \ref{size} we may also add the restrictions $a \leq b$ and $b < ce$, and from Proposition \ref{type-2} we can also require that $a,b$ be coprime.  Since
\begin{align*}
(ade) (acd) (ab)^{1/2} &\ll (ade) (acd) b\\ 
&\ll (ade) (acd) (ce) \\
&= (acde)^2 \\
&\ll N^2
\end{align*}
we see that one of the quantities $ade, acd, ab$ must be at most $O(N^{4/5})$ (cf. Section \ref{upper-sec}).  As we shall soon see, the ability to take one of these quantities to be significantly less than $N$ allows us to avoid the inefficiencies in the Brun-Titchmarsh inequality \eqref{brun} that led to a double logarithmic loss in the Type I case.  (Unfortunately, it does not seem that a similar trick is available in the Type II case.)

Let us first consider those quadruples with $ade \ll N^{4/5}$, which is the
easiest case.  For fixed $a,d,e$, $4acde-4a^2d-e$ traverses (a possibly
non-primitive) residue class modulo $4ade$.  As $ade \ll N^{4/5}$, there are no
primes in this class that are at least $N/2$ if the class is not primitive.  If
it is primitive, we may apply the Brun-Titchmarsh inequality \eqref{brun} to
bound the number of primes between $N/2$ and $N$ in this class by 
$O(\frac{N}{\phi(4ade) \log(N)})$, noting that $\log(N/4ade)$ is comparable to $\log N$.  Thus, we can bound this contribution to the left-hand side of \eqref{ppn} by
$$ \ll \frac{N}{\log N} \sum_{a,d,e: ade \ll N^{4/5}} \frac{1}{\phi(4acd)};$$
setting $m:= ade$ and bounding $\phi(4ade) \geq \phi(ade)$, we can bound this in turn by
$$
 \ll \frac{N}{\log N} \sum_{m \ll N^{4/5}} \frac{\tau_3(m)}{\phi(m)}
$$
where $\tau_3(m) := \sum_{a,d,e: ade=m} 1$.  Applying Lemma \ref{upper-crude}, we have
\begin{equation}\label{easy}
\sum_{m \ll N^{4/5}} \frac{\tau_3(m)}{\phi(m)} \ll \log^3 N,
\end{equation}
and so this contribution is acceptable.

Now we consider the case $acd \ll N^{4/5}$.  Here, we rewrite $4acde-4a^2d-e$ as $(4acd-1) e - 4a^2 d$, which then traverses a (possibly non-primitive) residue class modulo $4acd-1$.  Applying the Brun-Titchmarsh inequality as before, we may bound this contribution by
$$ \ll \frac{N}{\log N} \sum_{a,c,d: acd \ll N^{4/5}} \frac{1}{\phi(4acd-1)}$$
and hence (setting $m := 4acd-1$) by
$$ \ll \frac{N}{\log N} \sum_{m \ll N^{4/5}} \frac{\tau_3(m+1)}{\phi(m)},$$
so that it suffices to establish the bound
\begin{equation}\label{hard}
 \sum_{m \ll N^{4/5}} \frac{\tau_3(m+1)}{\phi(m)} \ll \log^3 N.
 \end{equation}
This is superficially similar to \eqref{easy}, but this time the summand is not multiplicative in $m$, and we can no longer directly apply Lemma \ref{upper-crude}.  To deal with this, we apply \eqref{phiq-bound} and bound \eqref{hard} by
$$ \ll \sum_{m \ll N^{4/5}} \sum_{d \mid m} \frac{\tau_3(m+1)}{dm};$$
writing $m = dn$, we can rearrange this as
$$ \ll \sum_{d \ll N^{4/5}} \frac{1}{d^2} \sum_{n \ll N^{4/5}/d} \frac{\tau_3(dn+1)}{n}.$$
Applying dyadic decomposition of the $d,n$ variables and using Proposition \ref{ab3}, we obtain \eqref{hard} as required.

Finally, we consider the case $ab \ll N^{4/5}$.  Here, we rewrite $4acde-4a^2d-e$ as $4abd - e$, and note that $e$ divides $a+b=ce$.  
If we fix $a,b$, there are thus at most $\tau(a+b)$ choices for $e$ (which also fixes $c$), and once one fixes such a choice, $4abd-e$ traverses a (possibly non-primitive) residue class modulo $4ab$.  Applying the Brun-Titchmarsh inequality again, we may bound this contribution by
$$ \ll \frac{N}{\log N} \sum_{a,b: ab \ll N^{4/5}; (a,b)=1} \frac{\tau(a+b)}{\phi(4ab)}.$$
Bounding $\phi(4ab) \geq \phi(ab)$ and using \eqref{phiq-bound}, we can bound this by
$$ \ll \frac{N}{\log N} \sum_{a,b: ab \ll N^{4/5}; (a,b)=1} \sum_{k \mid a} \sum_{l \mid b} \frac{\tau(a+b)}{abkl}.$$
Writing $a = km$, $b=ln$, we may bound this by
$$ \ll \frac{N}{\log N} \sum_{\substack{k,l,m,n: klmn \ll N^{4/5};\\ (k,l,m,n) = 1}} \frac{1}{k^2l^2mn} \tau(km+ln).$$
Dyadically decomposing in $k,l,m,n$ and using Proposition \ref{abcd-prop}, we see that this contribution is also $O(N \log^2 N)$.  The proof of \eqref{ppn} (and thus Theorem \ref{main}) is now complete.

\section{Solutions by polynomials}\label{solution}

We now prove Proposition \ref{res-class}.  We first verify that each of the sets is solvable by polynomials (which of course implies that any residue class contained in such classes are also solvable by polynomials). We first do this for the Type I sets.  In view of the $\pi^\I_n$ map (which clearly preserves polynomiality), it will suffice to find polynomials $a=a(n),\ldots,f=f(n)$ of $n$ that take values in $\N$ for sufficiently large $n$ in these sets, and such that $(a(n),\ldots,f(n)) \in \Sigma^\I_n$ for all $n$.  This is achieved as follows:

\begin{itemize}
\item If $n=-f \mod 4ad$, where $a,d,f \in \N$ are such that $f \mid 4a^2 d+1$, then we take 
$$ (a,b,c,d,e,f) := \left(a, \frac{n+f}{4ad}e - a, \frac{n+f}{4ad}, d, e,\frac{4a^2d+1}{e}\right).$$
\item If $n = -f \mod 4ac$ and $n = -{c}/{a} \mod f$, where $a,c,f \in \N$ are such that $(4ac,f)=1$, then we take
$$ (a,b,c,d,e,f) := \left(a, \frac{na+c}{f}, c, \frac{n+f}{4ac}, \frac{na+af+c}{fc}, f\right);$$
note from the hypotheses that $na+af+c$ is divisible by the coprime moduli $f$
and $c$, and is thus also divisible by $fc$.
\item If $n = -f \mod 4cd$ and $n^2 = -4c^2d \mod f$, where $c,d,f,q \in \N$ are such that $(4cd,f)=1$, then we take
$$ (a,b,c,d,e,f) := \left(\frac{n+f}{4cd}, \frac{n^2+4c^2d+nf}{4cdf}, c, d, \frac{(n+f)^2+4c^2d}{4c^2df}, f\right);$$
note from the hypotheses that $(n+f)^2+4c^2d$ is divisible by the coprime moduli $4c^2d$ and $f$, and is thus also divisible by $4c^2df$.
\item If $n = -{1}/{e} \mod 4ab$, where $a,b,e \in \N$ are such that $e \mid a+b$ and $(e,4ab)=1$, then we take
$$ (a,b,c,d,e,f) := \left(a,b, \frac{a+b}{e}, \frac{ne+1}{4ab}, e, 4a\frac{a+b}{e} \frac{ne+1}{4ab}-n\right)$$
\end{itemize}
One easily verifies in each of these cases that one has an $\N$-point of $\Sigma^\I_n$ for $n$ large enough.

Now we turn to the Type II case.  We use the same arguments as before, but using $\Sigma^\II_n$ in place of $\Sigma^\I_n$ of course:

\begin{itemize}
\item If $n=-e \mod 4ab$, where $a,b,e \in \N$ are such that $e \mid a+b$ and $(e,4ab)=1$, then we take
$$ (a,b,c,d,e,f) := \left(a,b,\frac{a+b}{e},\frac{n+e}{4ab}, e, \frac{a+b}{e} \frac{n+e}{b} - 1\right).$$
\item If $n=-4a^2d \mod f$, where $a,d,f \in \N$ are such that $4ad \mid f+1$, then we take
$$ (a,b,c,d,e,f) := \left(a,\frac{f+1}{4ad} \frac{n+4a^2d}{f}-a,\frac{f+1}{4ad},d,\frac{n+4a^2d}{f}, f\right).$$
\item If $n=-4a^2d-e \mod 4ade$, where $a,d,e \in \N$ are such that $(4ad,e)=1$, then we take
$$ (a,b,c,d,e,f) := \left(a,\frac{n+e}{4ad},\frac{n+4a^2 d + e}{4ade},d,e,\frac{n+4a^2d}{e}\right).$$
\end{itemize}
Again, one easily verifies in each of these cases that one has an $\N$-point of $\Sigma^\II_n$ for $n$ large enough.

Now we establish the converse claim.  Suppose first that we have a primitive residue class $q \mod r$ that can be solved by polynomials, then we have
$$ \frac{4}{p} = \frac{1}{x} + \frac{1}{y} + \frac{1}{z}$$
for all sufficiently large primes $p$ in this class, where $x=x(p), y=y(p), z=z(p)$ are polynomials of $p$ that take natural number values for all large $p$ in this class.  Note that (depending on whether the constant coefficient of $x(p)$ is nonzero or not) one either has $p|x(p)$ for all $p$, or $p \not | x(p)$ for all sufficiently large $p$. Similarly for $y(p)$ and $z(p)$.  Thus, after permuting, we may assume that we are either in the Type I case where $p|x(p)$ and $p \not|y(p),z(p)$ for all sufficiently large $p$ in the class, or else in the Type II case where $p\not | x(p)$ and $p | y(p),z(p)$ for all sufficiently large $p$ in the class. 

Suppose first we are in the Type I case.  For all sufficiently large $p$, we either have $y(p) \leq z(p)$ for all $p$, or $y(p) \geq z(p)$ for all $p$; by symmetry we may assume the latter.

Applying Proposition \ref{type-1}, we see that
$$ (x,y,z) = (abdp, acd, bcd)$$
for some $\N$-point $(a,\ldots,f) = (a(p),\ldots,f(p))$ in $\Sigma^\I_p$ with $a(p), b(p), c(p)$ having no common factor.  In particular, $d=d(p)$ is the greatest common divisor of $x(p), y(p), z(p)$.  On the other hand, if we let $\tilde d \in \Q[t]$ be the monic greatest common divisor of $x,y,z$, then there are natural numbers $C_1,C_2$ such that $C_1 \tilde d(p)$ is always an integer and $C_1 \tilde d(p)$ divides $C_2 x(p), C_2 y(p), C_2 z(p)$, and by the Euclidean algorithm we know that there is also a natural number $C_3$ such that $C_1 C_3 \tilde d(p)$ is an integer combination of $x(p), y(p), z(p)$.  From this we see that $d(p)$ is of the form $q(p) \tilde d(p)$ where $q(p)$ is a rational number that takes on only finitely many values as $p$ varies.  For any given rational $q$, the question of whether $q \tilde d(p)$ is an integer, and whether $q \tilde d(p)$ divides $x(p),y(p),z(p)$, can be determined in terms of finitely many residue classes $p \operatorname{ mod } r$ of $p$ (note that $\frac{x(p)}{q\tilde d(p)}, \frac{y(p)}{q\tilde d(p)}, \frac{z(p)}{q\tilde d(p)}$ are polynomials in $p$ with rational coefficients).  Thus, one can partition the original residue class of $p$ into finitely many subclasses, such that on each such class, $q(p)=q$ is independent of $p$.  We now pass to an arbitrary such subclass (eliminating the non-primitive classes, as these only contain at most one prime), so that $d(p)$ is now a polynomial function of $p$.  Dividing out by $d$ and repeating these arguments, we conclude (after passing to further subclasses if necessary) that $a = a(p)$, $b = b(p)$, and $c = c(p)$ are also polynomials in $p$ for sufficiently large $p$ in the subclass.  Applying the identities \eqref{I-1}-\eqref{I-9} we also see that $e = e(p)$ and $f=f(p)$ are polynomials in $p$ for sufficiently large $p$.  It will then suffice to show that all subclasses obtained in this fashion lie in a residue class from one of the Type I families in Proposition \ref{res-class}.

From Lemma \ref{size} we have $a(p) c(p) d(p) = O(p)$ and $f(p) = O(p)$ for all $p$, which implies that at least two of the polynomials $a(p), c(p), d(p)$ must be constant in $p$, and that $f(p)$ has degree at most $1$ in $p$.  We now divide into several cases.

First suppose that $a, d$ are independent of $p$.  By \eqref{I-7} this forces $e, f$ to be independent of $p$ as well, and $f$ divides $4a^2d+1$.  By \eqref{I-6} we have
$$ p = -f \mod 4ad$$
for all sufficiently large primes $p$ in the given subclass, and the claim follows in this case.

Now suppose that $a, c$ are independent of $p$, and $f$ has degree $0$ (i.e. is also independent of $p$).  Then from \eqref{I-6} we have $p = -f \mod 4ac$, and from \eqref{I-8} we have $p = -{c}/{a} \mod f$; since $p$ is a large prime this also forces $(4ac,f)=1$, and the claim follows.

Now suppose instead that $a, c$ are independent of $p$, and $f$ has degree $1$ (and thus grows linearly in $p$).  By Lemma \ref{size}, $b, e$ are then bounded and thus constant in $p$.  From \eqref{I-2} we have $e \mid a+b$, and from \eqref{I-1} we have $p = -{1}/{e} \mod 4ab$.  As $p$ is an arbitrarily large prime, this forces $(4ab,e)=1$, and the claim follows.
 
Next, suppose that $c, d$ are independent of $p$, and $f$ has degree $0$.  Then from \eqref{I-6} one has $p = -f \mod 4cd$, which in particular forces $(4cd,f)=1$.  From \eqref{I-9} one has $p^2 = -4c^2 d \mod f$, and the claim follows.

Finally, suppose that $c, d$ are independent of $p$, and $f$ has degree $1$.  By \eqref{I-9}, $f(p)$ divides $p^2 + 4c^2 d$ for all large primes $p$ in the primitive residue subclass.  Applying the Euclidean algorithm, we conclude that $f$ in fact divides $p^2+4c^2d$ \emph{as a polynomial} in $p$.  But as $c,d$ are positive, $p^2+4c^2 d$ is irreducible over the reals, a contradiction.  This concludes the treatment of the Type I case.

Now suppose we are in the Type II case.  Arguing as in the Type I case, we obtain an $\N$-point $(a,\ldots,f) = (a(p),\ldots,f(p))$ in $\Sigma^\II_p$ for all sufficiently large primes $p$ in this class, and obeying the bounds in Lemma \ref{size}, and after partitioning the set of such large primes $p$ into a finite number of primitive subclasses, one has $a(p),\ldots,f(p)$ all depending in a polynomial fashion on $p$ in each subclass.  

We now work with an individual subclass and show that all sufficiently large primes $p$ in this subclass lie in a residue class in one of the Type II families in Proposition \ref{res-class}. 
From Lemma \ref{size} we have $a(p) c(p) d(p) e(p) = O(p)$, and so three of these polynomials $a(p), c(p), d(p), e(p)$ must be independent of $p$.

Suppose first that $a, c, e$ are independent of $p$.  By \eqref{I-2}, $b$ is independent of $p$ also, and $e \mid a+b$.  By \eqref{II-1}, $p = -e \mod 4ab$, and thus $(e,4ab)=1$, and the claim then follows from Dirichlet's theorem.

Now suppose that $a, c, d$ are independent of $p$.  By \eqref{II-6}, $f$ is independent of $p$ also, and $4ad \mid f+1$.    From \eqref{II-7} one has $p=-4a^2 d \mod f$, and the claim follows.

Next, suppose $a, d, e$ are independent of $p$.  By \eqref{II-4} one has $p = -4a^2d -e \mod 4ade$, which implies $(4ad,e)=1$, and the claim follows.

Finally, suppose $c,d,e$ are independent of $p$.  By \eqref{II-2} this forces $a,b$ to be bounded, and hence also independent of $p$; and so this case is subsumed by the preceding cases.

\section{Lower bounds III}\label{lower-3}

\subsection{Generation of solutions}

We begin the proof of Theorem \ref{many-solutions}; the method of proof will be a generalisation of that in Section \ref{lower-sec}.  For the rest of this section, $m$ and $k$ are fixed, and all implied constants in asymptotic notation are allowed to depend on $m,k$.  We assume that $N$ is sufficiently large depending on $m,k$.

In the $m=4,k=3$ case, Type II solutions were generated by the ansatz
$$ (t_1,t_2,t_3) = (abd, acdn, bcdn)$$
for various quadruples $(a,b,c,d)$ (or equivalently, quadruples $(a,c,d,e)$, setting $b := ce-a$); see \eqref{pi2-n}.  We will use a generalisation of this ansatz for higher $k$; for instance, when $k=4$ we will construct solutions of the form
$$ (t_1,t_2,t_3,t_4) = (b x_{12} x_{123} x_{124} x_{1234}, x_{12} x_{23} x_{24} x_{123} x_{124} x_{234} x_{1234} n, b x_{23} x_{123} x_{234} x_{1234} n, b x_{24} x_{124} x_{234} x_{1234} n)$$
for various octuples $(b,x_{12}, x_{23}, x_{24}, x_{123}, x_{124}, x_{234}, x_{1234})$, or equivalently, using octuples 
$$(x_{12}, x_{23}, x_{24}, x_{123}, x_{124}, x_{234}, x_{1234}, e),$$
and setting
$$ b = e x_{23} x_{24} x_{234} - x_{12} x_{24} x_{124} - x_{12} x_{23} x_{123}.$$
More generally, we will generate Type II solutions via the following lemma.

\begin{lemma}[Generation of Type II solutions]  Let ${\mathcal P}$ denote the set $2^{k-1}-1$-element set
$$ {\mathcal P} := \{ I \subset \{1,\ldots,k\}: 2 \in I; I \neq \{2\} \}.$$
Let $(x_I)_{I \in {\mathcal P}}$ be a tuple of natural numbers, and let $e$ be another natural number, obeying the inequalities
\begin{align}
\frac{1}{2m} N \leq e \prod_{I \in {\mathcal P}} x_I &\leq \frac{1}{m} N\label{eprod} 
\end{align}
and
\begin{equation}\label{xi-bound}
1 <  x_I \leq N^{1/2^{k+2}}
\end{equation}
whenever $I \in {\mathcal P}$.  Suppose also that the quantity
\begin{equation}\label{square}
 w := \prod_{I \in {\mathcal P}: I \neq \{1,2\}} x_I
\end{equation}
is square-free.  Set
\begin{align}
 b &:= e \prod_{I \in {\mathcal P}: 1 \not \in I} x_I - \sum_{j=3}^k \prod_{I \in {\mathcal P}: j \not \in I} x_I\label{x13}\\
 t_1 &:= b \prod_{I \in {\mathcal P}: 1 \in I} x_I \label{t1-def} \\
 n &:= m t_1 - e \label{ndef} \\
 t_2 &:= n \prod_{I \in {\mathcal P}} x_I \label{t2-def}
\end{align}
and
\begin{equation}\label{tj-def}
t_j := b\, n \prod_{I \in {\mathcal P}: j \in I} x_I.
\end{equation}
Then $n$ is a natural number with $n \leq N$, and $(t_1,\ldots,t_k)$ is a Type II solution for this value of $n$.  Furthermore, each choice of $(x_I)_{I \in {\mathcal P}}$ and $e$ generates a distinct Type II solution.
\end{lemma}

\begin{remark} In the $m=4,k=3$ case, the parameters $x_I$ are related to the coordinates $(a,b,c,d,e,f)$ appearing in Proposition \ref{type-2} by the formula
$$
(a,b,c,d,e,f) = (x_{12}, b, x_{23}, x_{123}, e, 4 x_{12} x_{23} x_{123} - 1);$$
however, the constraint that $a,b,c$ have no common factor and $abd$ is coprime to $n$ has been replaced by the slightly different criterion that $d$ is squarefree, which turns out to be more convenient for obtaining lower bounds (note that the same trick was also used to prove \eqref{2n-lower}).
  Parameterisations of this type have appeared numerous times in the previous
  literature (see \cite{gupta, hall, ruzsa, els}, or indeed Propositions
  \ref{type-1}, \ref{type-2}), though because most of these parameterisations
  were focused on dealing with \emph{all} solutions of a given type, as opposed to an easily countable subset of solutions, there were more parameters $x_I$ (indexed by all non-empty subsets of $\{1,\ldots,k\}$, not just the ones in ${\mathcal P}$), and there were some coprimality conditions on the $x_I$ rather than square-free conditions.
\end{remark}

\begin{proof}  Let the notation be as in the lemma.    Then from \eqref{xi-bound} one has
$$  \sum_{j=3}^k \prod_{I \in {\mathcal P}: j \not \in I} x_I \leq (k-2) N^{2^{k-2}/2^{k+2}} \ll N^{1/16}$$
while since
$$ \prod_{I \in {\mathcal P}} x_I \ll N^{2^{k-1}/2^{k+2}} \ll N^{1/8}$$
we see from \eqref{eprod} that
$$ e \gg N^{7/8}.$$
From \eqref{x13} we then have that
$$ \frac{1}{2} e \prod_{I \in {\mathcal P}: 1 \not \in I} x_I \leq b \leq e \prod_{I \in {\mathcal P}: 1 \not \in I} x_I $$
and thus by \eqref{t1-def}
$$ \frac{1}{2} e \prod_{I \in {\mathcal P}} x_I \leq t_1 \leq e \prod_{I \in {\mathcal P}} x_I$$
and thus by \eqref{ndef} (noting that $m \geq 4$)
$$ \frac{1}{4} m e \prod_{I \in {\mathcal P}} x_I \leq n \leq m e \prod_{I \in {\mathcal P}} x_I.$$
These bounds ensure that $b, n, t_1,\ldots,t_k$ are natural numbers with $n
\leq N$, and with $t_2,\ldots,t_k$ divisible by $n$.  Dividing \eqref{x13} by
$b\, n\prod_{I \in {\mathcal P}} x_I$ and using \eqref{t1-def}, \eqref{t2-def}, \eqref{tj-def}, we conclude that
$$ 
\frac{1}{t_2} = \frac{e}{nt_1} - \sum_{j=3}^k \frac{1}{t_j};$$
applying \eqref{ndef} one concludes that $(t_1,\ldots,t_k)$ is a Type II solution.

It remains to demonstrate that each choice of $(x_I)_{I \in {\mathcal P}}$ and $e$ generates a distinct Type II solution, or equivalently that the Type II solution $(t_1,\ldots,t_k)$ uniquely determines $(x_I)_{I \in {\mathcal P}}$ and $e$.  To do this, first observe from \eqref{m/nsumofk} that $(t_1,\ldots,t_k)$ determines $n$, and from \eqref{ndef} we see that $e$ is determined also.  Next, observe from \eqref{t1-def}, \eqref{t2-def}, \eqref{tj-def} that for any $3 \leq j \leq k$, one has
\begin{equation}\label{2j1}
\frac{t_2 t_j}{n^2 t_1} = \left(\prod_{I \in {\mathcal P}: j \in I; 1 \not \in I} x_I\right)^2 \left(\prod_{I \in {\mathcal P}: j \in I \hbox{ XOR } 1 \not \in I} x_I\right)
\end{equation}
where XOR denotes the exclusive or operator; in particular, the left-hand side is necessarily a natural number.  Note that all the factors $x_I$ appearing on the right-hand side are components of the square-free quantity $w$ given by \eqref{square}.  We conclude that $(\prod_{I \in {\mathcal P}: j \in I; 1 \not \in I} x_I)^2$ is the largest perfect square dividing $\frac{t_2 t_j}{n^2 t_1}$.  We conclude that the Type II solution $(t_1,\ldots,t_k)$ determines all the products
\begin{equation}\label{xii}
 \prod_{I \in {\mathcal P}: j \in I; 1 \not \in I} x_I
\end{equation}
for $3 \leq j \leq k$.  Note (from the square-free nature of $w$) that the
$x_I$ with $1 \not \in I$ are all coprime.  Taking the greatest common divisor
of the \eqref{xii} for all $3 \leq j \leq k$, we see that the Type II solution
determines $x_{\{2,3,\ldots,k\}}$.  Dividing this quantity out from all the
expressions \eqref{xii}, and then taking the greatest common divisor of the resulting quotients for $4 \leq j \leq k$, one recovers $x_{\{2,4,\ldots,k\}}$; a similar argument gives $x_I$ for any $I \in {\mathcal P}$ with $1 \not \in I$ of cardinality $k-3$.  Dividing out these quantities and taking greatest common divisors again, one can then recover $x_I$ for any $I \in {\mathcal P}$ with $1 \not \in I$ of cardinality $k-4$; continuing in this fashion we can recover all the $x_I$ with $I \in {\mathcal P}$ and $1 \not \in I$.

Returning to \eqref{2j1}, we can then recover the products $\prod_{I \in {\mathcal P}: 1, j \in I} x_I$ for all $3 \leq j \leq k$.  Taking greatest common divisors iteratively as before, we can then recover all the $x_I$ with $I \in {\mathcal P}$ and $1 \in I$, thus reconstructing all of the data $(x_I)_{I \in {\mathcal P}}$ and $e$, as claimed.
\end{proof}

In view of this above lemma, we see that to prove \eqref{many-1}, it suffices to show that the number of tuples
$( (x_I)_{I \in {\mathcal P}}, e )$ obeying the hypotheses of the lemma is at least $c N (\log N)^{2^{k-1}-1}$ for an absolute constant $c>0$.

Observe that if we fix $x_I$ with $I \in {\mathcal P}$ obeying \eqref{xi-bound} and with the quantity $w$ defined by \eqref{square}, then there are 
$$\gg \frac{N}{\prod_{I \in {\mathcal P}} x_I}$$ 
choices of $e$ that obey \eqref{eprod}.  Thus, noting that $\mu^2(w) \geq \mu^2(\prod_{I \in {\mathcal P}} x_I)$, 
the number of tuples obeying the hypotheses of the lemma is
\begin{equation}\label{mush}
\gg N \sum_* \frac{\mu^2(\prod_{I \in {\mathcal P}} x_I)}{\prod_{I \in {\mathcal P}} x_I},
\end{equation}
where the sum $\sum_*$ ranges over all choices of $(x_I)_{I \in {\mathcal P}}$ obeying the bounds \eqref{xi-bound}.  
To estimate \eqref{mush}, we make use of
\cite[Theorem 6.4]{elsholtz:2001}, which we restate as a lemma:

\begin{lemma}{\label{squarefree-sum}}  Let $l \geq 1$, and for each $1 \leq i \leq l$, let $\alpha_i < \beta_i$ be positive real numbers.  Then
\begin{equation}
\sum_{N^{\alpha_i} \leq n_i \leq N^{\beta_i} \hbox{ for all } 1 \leq i \leq l} 
\frac{\mu^2(n_1 \cdots n_l)}{n_1 \cdots n_l} 
\gg_l (\log N)^l \prod_{i=1}^l(\beta_i - \alpha_i),
\end{equation}
for $N$ sufficiently large depending on $l$ and the $\alpha_1,\ldots,\alpha_l,\beta_1,\ldots,\beta_l$.
\end{lemma}

From this lemma (and noting that there are $2^{k-1}-1$ parameters $x_I$ in the sum $\sum_*$) we see that 
\begin{equation}\label{mush-2}
\sum_* \frac{\mu^2(\prod_{I \in {\mathcal P}} x_I)}{\prod_{I \in {\mathcal P}} x_I} \gg \log^{2^{k-1}-1} N;
\end{equation}
inserting this into \eqref{mush} we obtain the claim.

Now we prove \eqref{many-2}.  As in Section \ref{lower-sec}, the arguments are similar to those used to prove \eqref{many-1}, but with the additional input of the Bombieri-Vinogradov inequality.

As in the proof of \eqref{many-1}, it suffices to obtain a lower bound (in this case, $c {N (\log N)^{2^{k-1}-2}}/{\log \log N}$ for some $c>0$) on the number of tuples
$( (x_I)_{I \in {\mathcal P}}, e )$, but now with the additional constraint that the quantity
\begin{align*}
p := mt_1 - e = m b \prod_{I \in {\mathcal P}: 1 \in I} x_I - e
\end{align*}
is prime.  

Suppose we fix $(x_I)_{I \in {\mathcal P}}$ obeying \eqref{xi-bound} with $w$ squarefree.  We may write
$$ p = qe + r$$
where
\begin{equation}\label{qmash}
 q := m \prod_{I \in {\mathcal P}} x_I - 1
 \end{equation}
and
$$ r := - m \prod_{I \in {\mathcal P}: 1 \in I} x_I \sum_{j=3}^k \prod_{I \in {\mathcal P}: j \not \in I} x_I.$$
Thus as $e$ varies in the range given by \eqref{eprod}, $qe+r$ traces out an arithmetic progression of spacing $q$ whose convex hull contains $[0.6 N, 0.9 N]$ (say).  Thus, every prime $p$ in this interval $[0.6 N, 0.9 N]$ that is congruent to $r \mod q$ will provide an $e$ that will give a Type II solution with $n=p$ prime, and different choices of $(x_I)_{I \in {\mathcal P}}$ and $p$ will give different Type II solutions.

For fixed $(x_I)_{I \in {\mathcal P}}$, if $r$ is coprime to $q$, then we see from \eqref{dnq} (and estimating $\li(x) = (1+o(1)) {x}/{\log x}$) that the number of such $p$ is at least
$$ \geq c \frac{N}{\log N \phi(q)} - D(0.6 N; q) - D(0.9 N; q)$$
for some absolute constant $c>0$.  It thus suffices to show that
\begin{equation}\label{many-2a}
 \sum_* \mu^2(w) 1_{(r,q)=1} \frac{N}{\log N \phi(q)} \gg \frac{N (\log N)^{2^{k-1}-2}}{\log \log N}
\end{equation}
and
\begin{equation}\label{many-2b}
 \sum_* D(cN;q) = o\left( \frac{N (\log N)^{2^{k-1}-2}}{\log \log N} \right)
\end{equation}
for $c=0.6, 0.9$.

We first show \eqref{many-2a}.  Since $\li(N/100)$ is comparable to $N/\log N$, and $\phi(q) \leq q \ll w$, we may simplify \eqref{many-2a} as
\begin{equation}\label{many-2c}
 \sum_* \frac{\mu^2(w)}{\prod_{I \in {\mathcal P}} x_I} 1_{(r,q)=1} \gg \frac{(\log N)^{2^{k-1}-1}}{\log \log N}.
\end{equation}
The expression on the left-hand side is similar to \eqref{mush}, but now one also has the additional constraint $1_{(r,q)=1}$.  To deal with this constraint, we restrict the ranges of the $x_I$ parameters somewhat to perform an averaging in the $x_{\{1,2\}}$ parameter (taking advantage of the fact that this parameter does not appear in the $\mu^2(w)$ term).  More precisely, we restrict to the ranges where
\begin{equation}\label{x-bang}
 x_I \leq N^{1/2^{100k}}
\end{equation}
(say) for $I \neq \{1,2\}$, and
\begin{equation}\label{y-bang}
x_{\{1,2\}} \leq N^{1/2^{k+2}}.
\end{equation}
We now analyse the constraint that $r$ and $q$ are coprime.  We can factor
$$ r = - m x_{\{1,2\}}^2 s$$
where
$$ s := \left(\prod_{I \in {\mathcal P}: 1 \in I; I \neq \{1,2\}} x_I\right) \sum_{j=3}^k \prod_{I \in {\mathcal P}: j \not \in I; I \ne \{1,2\}} x_I;$$
the point is that $s$ does not depend on $x_{\{1,2\}}$.  Since $q+1$ is divisible by $m x_{\{1,2\}}$, we see that $m x_{\{1,2\}}^2$ is coprime to $q$, and thus $(q,r)=1$ iff $(q,s)=1$.  We can write $q = u x_{\{1,2\}} - 1$, where $u := m \prod_{I \in {\mathcal P}: I \neq \{1,2\}} x_I$, and so $(q,r)=1$ iff $(u x_{\{1,2\}}-1,s)=1$.

We may replace $s$ here by the largest square-free factor $s'$ of $s$.  If we then factor $s' = vy$, where $v := (s',u)$ and $y := s'/v$, then $u x_{\{1,2\}}-1$ is already coprime to $v$, and so we conclude that $(q,r)=1$ iff $(u x_{\{1,2\}}-1,y)=1$.

Fix $x_I$ for $I \neq \{1,2\}$.  By construction, $u$ and $y$ are coprime, and so the constraint $(u x_{\{1,2\}}-1,y)=1$ restricts $x_{\{1,2\}}$ to $\phi(y)$ distinct residue classes modulo $y$.  Since
$$ y \leq s \ll N^{1/2^{90k}}$$
(say) thanks to \eqref{x-bang}, we conclude that 
$$ \sum_{x_{\{1,2\}} \leq N^{1/2^{k+2}}} \frac{1_{(q,r)=1}}{x_{\{1,2\}}} \gg \frac{\phi(y)}{y} \log N.$$
Using the crude bound \eqref{phi-lower}, we may lower bound ${\phi(y)}/{y} \gg {1}/{\log \log N}$.  (It is quite likely that by a finer analysis of the generic divisibility properties of $y$, one can remove this double logarithmic loss, but we will not attempt to do so here.)  We may thus lower bound the left-hand side of \eqref{many-2c} by
$$ \frac{\log N}{\log \log N} \sum_{**} \frac{\mu^2(w)}{w},$$
where $\sum_{**}$ sums over all $x_I$ for $I \neq \{1,2\}$ obeying \eqref{x-bang}.  But by Lemma \ref{squarefree-sum} we have
$$ \sum_{**} \frac{\mu^2(w)}{w} \gg (\log N)^{2^{k-1}-2},$$
and the claim \eqref{many-2c} follows.
 
Finally, we show \eqref{many-2b}.  Observe that each $q$ can be represented in the form \eqref{qmash} in at most $\tau_{2^{k-1}-1}(q+1)$ different ways; also, from \eqref{xi-bound} we have $q \ll N^{2^{k-1}/2^{k+2}} = N^{1/8}$.  We may thus bound the left-hand side of \eqref{many-2b} by
$$
\sum_{q \ll N^{1/8}} D(cN;q) \tau_{2^{k-1}-1}(q+1).$$
From the Bombieri-Vinogradov inequality \eqref{bombieri} and the trivial bound $D(cN;q) \ll N/q$ one has
$$
\sum_{q \ll N^{1/8}} q D(cN;q)^2 \ll_A N \log^{-A} N$$
for any $A>0$, while from Lemma \ref{upper-crude} (and shifting $q$ by $1$) one has
$$ \sum_{q \ll N^{1/8}} \frac{\tau_{2^{k-1}-1}(q+1)^2}{q} \ll \log^{O(1)} N.$$
The claim then follows from the Cauchy-Schwarz inequality (taking $A$ large enough).  The proof of Theorem \ref{many-solutions} is now complete.

\appendix

\section{Some results from number theory}

In this section we record some well-known facts from number theory that we will need throughout the paper.  We begin with a crude estimate for averages of multiplicative functions.

Now we record some asymptotic formulae for the divisor function $\tau$.  From the Dirichlet hyperbola method we have the asymptotic
\begin{equation}\label{tau-1}
 \sum_{n \leq N} \tau(n) = N \log N + O(N)
\end{equation}
(see e.g. \cite[\S 1.5]{iwaniec}). More generally, we have
\begin{equation}\label{tau-k}
 \sum_{n \leq N} \tau_k(n) = N \log^{k-1} N + O_k(N \log^{k-2} N)
\end{equation}
for all $k \geq 1$, where $\tau_k(n) := \sum_{d_1,\ldots,d_k: d_1 \ldots d_k = n} 1$.  Indeed, the left-hand side of \eqref{tau-k} can be rearranged as
$$ \sum_{d_1 \leq N} \sum_{d_2 \leq N/d_1} \ldots \sum_{d_k \leq N/d_1 \ldots d_{k-1}} 1$$
and the claim follows by evaluating each of the summations in turn.

We can perturb this asymptotic:

\begin{lemma}[Crude bounds on sums of multiplicative functions]\label{upper-crude}  Let $f(n)$ be a multiplicative function obeying the bounds
$$ f(p) = m + O(\frac{1}{p})$$
for all primes $p$ and some integer $m \geq 1$, and
$$ |f(p^j)| \ll j^{O(1)}$$
for all primes $p$ and $j > 1$.  Then one has
$$ \sum_{n \leq N} f(n) \ll_m N \log^{m-1} N$$
for $N$ sufficiently large depending on $m$; from this and summation by parts we have in particular that
$$ \sum_{n \leq N} \frac{f(n)}{n} \ll_m \log^{m} N$$
If $f$ is non-negative, we also have the corresponding lower bound
$$ \sum_{n \leq N} f(n) \gg_m N \log^{m-1} N$$
and hence
$$ \sum_{n \leq N} \frac{f(n)}{n} \gg_m \log^{m} N$$
\end{lemma}

One can of course get much better estimates by contour integration methods (and these estimates also follow without much difficulty from the more general results in \cite{halb}), but the above crude bounds will suffice for our purposes.

\begin{proof}  We allow all implied constants to depend on $m$.  By M\"obius inversion, we can write
$$ f(n) = \sum_{d \mid n} \tau_{m}(d) g(\frac{n}{d})$$
where $g$ is a multiplicative function obeying the bounds
$$ g(p) = O(\frac{1}{p})$$
and 
$$ |g(p^j)|\ll j^{O(1)}$$
for all $j > 1$.  In particular, the Euler product
$$ \sum_{n=1}^\infty \frac{|g(n)|}{n} = \prod_p \left(1 + \frac{|g(p)|}{p} +
  \sum_{j=2}^\infty \frac{|g(p^j)|}{p^j}\right) = 
\prod_p \left(1 + O\left(\frac{1}{p^2}\right)\right)$$
is absolutely convergent.

We may therefore write $\sum_{n \leq N} f(n)$ as
\begin{equation}\label{loo}
 \sum_{k \leq N} g(k) \sum_{d \leq N/k} \tau_{m}(d).
 \end{equation}
Applying \eqref{tau-k}, we conclude
$$ |\sum_{n \leq N} f(n)| \ll \sum_{k \leq N} \frac{|g(k)|}{k} N \log^{m-1} N$$
and the upper bound follows from the absolute convergence of $\sum_{n=1}^\infty {|g(n)|}/{n}$.

Now we establish the lower bound.  By zeroing out $f$ at various small primes $p$ (and all their multiples), we may assume that $f(p^j)=g(p^j)=0$ for all $p \leq w$ for any fixed threshold $w$.  By making $w$ large enough, we may ensure that
$$ 1 - \sum_{n=2}^\infty \frac{|g(n)|}{n} > 0.$$
If we then insert the bound \eqref{tau-k} into \eqref{loo} we obtain the claim.
\end{proof}

As a typical application of Lemma \ref{upper-crude} we have
\begin{equation}\label{tau-2}
\sum_{n \leq N} \tau^k(n) \ll_k N \log^{2^k-1} N
\end{equation}
for any $N>1$ and $k \geq 1$,
(see also \cite{Mardjanichvili:1939}).

To study some more detailed distribution of divisors and prime divisors
we recall the \emph{Tur\'{a}n-Kubilius inequality}
for additive functions.
A function $w$ is called additive, if $w(n_1 n_2)=w(n_1)+w(n_2)$, 
whenever $\gcd(n_1,n_2)=1$.

\begin{lemma}[Tur\'{a}n-Kubilius inequality 
(see \cite{schwarzandspilker}, page 20)]{\label{turan-kubilius}}
Let $w:\N \rightarrow \R$ denote an arithmetic function which is additive (thus $w(nm)=w(n)+w(m)$ whenever $n,m$ are coprime).
Let $A(N):=\sum_{p^k\leq N} {w(p^k)}/{p^k}$ and 
$D^2(N):=\sum_{p^k\leq N} {|w(p^k)|^2}/{p^k}$.
For every $N\geq 2$ and for any additive function $w$ the following 
inequality holds:
\[ \sum_{n \leq N} |w(n)-A(N)|^2\leq 30 N D^2(N).\]
(Here $\sum_{p^k}$ denotes the sum over all prime powers.)
\end{lemma}

\begin{example} Let $\omega(n)$ denote the number of distinct prime factors of $n$, then
$A(N)=\sum_{p^k\leq N} {\omega(p^k)}/{p^k}=\log \log N+ O(1)$ and 
$D^2(N)=\sum_{p^k\leq N} {\omega(p^k)^2}/{p^k}=A(N)=\log \log N + O(1)$.
The Tur\'an-Kubilius inequality then gives
\[ \sum_{n \leq N} |\omega(n)-\log \log N|^2\leq 30 N \log \log N +O(N).\]
In particular, if $\xi(n) \to \infty$ as $n \to \infty$, then one has $|\omega(n) - \log \log n| \leq \xi(n) \sqrt{\log\log n}$ for all $n$ in a set of integers of density $1$.  For more details see \cite{tenenbaum}.
\end{example}

From \eqref{tau-1} one might guess the heuristic
\begin{equation}\label{tau-heuristic}
\tau(n) \approx \log n
\end{equation}
\emph{on average}. But it follows from the 
Tur\'{a}n-Kubilius inequality
that for ``typical'' $n$, the number of divsors is about 
$2^{\log \log n}=(\log n)^{\log 2}$, which is
considerably smaller, and that a small number of integers with an 
exceptionally large
number of divisors heavily influences this average. The influence of these
integers with a very large number of divsiors dominates even more for higher
moments.
The extremal cases heuristically consist of many small prime factors, 
and the following
``divisor bound'' holds
\begin{equation}\label{tau-divisor}
\tau(n) \leq 2^{(1+o(1)) \frac{\log n}{\log \log n} } = O(n^{\frac{1}{\log \log n}})
\end{equation}
for any $n \geq 1$; see \cite{ramanujan}.

The Tur\'{a}n-Kubilius type inequalities have been studied for shifted primes
as well. We make use of the following result of Barban 
(see Elliott \cite{elliott}, Theorem 12.10).

\begin{lemma}{\label{barban-turankubilius}}
A function $w: \N\rightarrow \R^+$ is said to be strongly additive
if it is additive and $w(p^k)=w(p)$ holds, for every prime power $p^k$, 
$k \geq 1$.
Let $w$ denote a real nonnegative strongly additive function.
Define $S(N) :=\sum_{p \leq N} {w(p)}/{(p-1)}$ and 
$\Lambda(N):=\max_{p \leq N} w(p)$. Suppose that $\Lambda(N)=o(S(N))$, as $N
\rightarrow \infty$. Then for any fixed $\varepsilon >0$, the prime density
\[ \nu_N(p; |w(p+1)-S(N)| > \varepsilon S(N)) \rightarrow 0 \text{ as } N
\rightarrow \infty.\]
The same holds for other shifts $p+a$, where $a \neq 0$.
\end{lemma}
The function $\omega(n)$ is strongly additive.
This lemma implies that for primes with relative prime density 1, 
$p+1$ contains about $\frac{1}{2}\log \log p$ primes of the 
form $1 \bmod 4$. To see
this one chooses $w(p)=1$ if $p\equiv 1 \bmod 4$, and $0$ otherwise.
In this example one has $S(N)\sim \frac{1}{2}\log \log N$ and $\Lambda(N)=1$.

We recall the quadratic reciprocity law
\begin{equation}\label{quadratic}
\left( \frac{m}{n} \right) \left( \frac{n}{m} \right) = (-1)^{(n-1)(m-1)/4}
\end{equation}
for all odd $m,n$, where $\left( \frac{m}{n} \right)$ is the Jacobi symbol, as well as the companion laws
\begin{equation}\label{quadratic-1}
\left( \frac{-1}{n} \right) = (-1)^{(n-1)/4}
\end{equation}
and
\begin{equation}\label{quadratic-2}
\left( \frac{2}{n} \right) = (-1)^{(n^2-1)/8}
\end{equation}
for odd $n$.

For any primitive residue class $a \mod q$ and any $N > 0$, let $\pi(N;q,a)$ denote the number of primes $p<N$ that are congruent to $a$ mod $q$.
We recall the \emph{Brun-Titchmarsh inequality} (see e.g. \cite[Theorem 6.6]{iwaniec})
\begin{equation}\label{brun}
\pi(N;q,a) \ll \frac{N}{\phi(q) \log \frac{N}{q}}
\end{equation}
for any such class with $N \geq q$.  This bound suffices for upper bound estimates on primes in residue classes.   Due to the $q$ in the denominator of $\log({N}/{q})$, it will only be efficient to apply this inequality when $q$ is much smaller than $N$, e.g. $q \leq N^c$ for some $c<1$.

The Euler totient function $\phi(q)$ in the denominator is also inconvenient; it would be preferable if one could replace it with $q$.  Unfortunately, this is not possible; the best bound on ${1}/{\phi(q)}$ in terms of $q$ that one has in general is
\begin{equation}\label{phi-lower}
\frac{1}{\phi(q)} \ll \frac{\log \log q}{q}
\end{equation}
(see e.g. \cite{rosser}).  Using this bound would simplify our arguments, but one would lose an additional factor of $\log \log N$ or so in the final estimates.  To avoid this loss, we observe the related estimate
\begin{equation}\label{phiq-bound}
\frac{1}{\phi(q)} \ll \frac{1}{q} \sum_{d \mid q} \frac{1}{d}.
\end{equation}
Indeed, we have
\begin{align*}
\frac{q}{\phi(q)} &= \prod_{p \mid q} \frac{p}{p-1} \\
&= \prod_{p \mid q} (1 + \frac{1}{p}) (1 + O(\frac{1}{p^2})) \\
&\ll \prod_{p \mid q} (1 + \frac{1}{p}) \\
&\leq \sum_{d \mid q} \frac{1}{d},
\end{align*}
and \eqref{phiq-bound} follows.  (One could restrict $d$ to be square-free here if desired, but we will not need to do so in this paper.)

The Brun-Titchmarsh inequality only gives upper bounds for the number of primes
in an arithmetic progression.
  To get lower bounds, we let $D(N;q)$ denote the quantity
\begin{equation}\label{dnq}
D(N;q) := \max_{(a,q)=1} \left|\pi(N;q,a) - \frac{\li(N)}{\phi(q)}\right|.
\end{equation}
where $\li(x) := \int_0^x {dt}/{\log t}$ is the Cauchy principal value of the
logarithmic integral.
The Bombieri-Vinogradov inequality (see e.g. \cite[Theorem 17.1]{iwaniec})
implies in particular
that
\begin{equation}\label{bombieri}
 \sum_{q \leq N^\theta} D(N;q) \ll_{\theta,A} N \log^{-A} N.
\end{equation}
We remark that the above inequality is usually phrased using the  summatory von Mangoldt
  function $\psi(N;q,a) = \sum_{n \leq N; n = a \mod q} \Lambda(n)$. A
  summation by parts converts  it to an estimate using the prime counting function; see \cite{Bruedern:1995}
for details. 

for all $0 < \theta < 1/2$ and $A > 0$.  Informally, this gives lower bounds on $\pi(N;q,a)$ on the average for $q$ much smaller than $N^{1/2}$.

\bibliographystyle{line}

\end{document}